\documentclass[preprint, 12pt]{elsarticle}

\usepackage{amssymb}
\usepackage{amsthm}
\usepackage[multiple]{footmisc}
\usepackage{mathtools}
\usepackage{enumitem}
\usepackage{hyperref}
\usepackage{graphicx}
\usepackage{caption}
\usepackage{subcaption}
\usepackage{wrapfig}
\usepackage{float}

\newcommand{\mbb}{\mathbb}

\newcommand{\ones}{{\bf{1}}}

\newcommand{\R}{\mathbb{R}}
\newcommand{\N}{\mathbb{N}}

\newcommand{\sgn}{\mbox{\textup{sgn}}}

\newcommand{\rank}{\mbox{\textup{rank\,}}}

\newcommand{\idx}{\mbox{\textup{idx}}}

\newcommand{\carr}{\mbox{\textup{carr}}}
\newcommand{\extcarr}{\mbox{\textup{ext carr\,}}}

\newcommand{\range}{\mbox{\textup{range\,}}}

\newcommand{\actionikone}{y_{i}^{k+1}}
\newcommand{\actionione}{y_{i}^{1}}
\newcommand{\actionitwo}{y_{i}^{2}}
\newcommand{\actionigammai}{y_{i}^{\gamma_i}}

\newcommand{\cl}{\mbox{\textup{cl\,}}}

\newcommand{\diag}{\textup{diag}\,}
\newcommand{\IR}{\textup{IR}}
\newcommand{\rOmega}{X^*}
\newcommand{\vzero}{{}\bf 0}

 \def\vzero{{\bf 0}}
 \def\vb{{\bf b}}

   \def\vp{{\bf p}}
 \def\vr{{\bf r}}

\def\vA{{\bf A}} \def\vB{{\bf B}} \def\vC{{\bf C}} \def\vD{{\bf D}}
\def\vE{{\bf E}} \def\vF{{\bf F}}  \def\vH{{\bf H}}
\def\vI{{\bf I}} \def\vJ{{\bf J}}  \def\vL{{\bf L}}
\def\vM{{\bf M}}   \def\vP{{\bf P}}
 \def\vR{{\bf R}}

  \def\calI{\mathcal{I}}
  \def\calL{\mathcal{L}}
  
\def\calP{\mathcal{P}}  
\def\calS{\mathcal{S}}  \def\calU{\mathcal{U}}
 \def\calW{\mathcal{W}} 
 
\newtheorem{theorem}{Theorem}
\newtheorem{proposition}[theorem]{Proposition}
\newtheorem{example}[theorem]{Example}
\newtheorem{remark}[theorem]{Remark}
\newtheorem{mydef}[theorem]{Definition}
\newtheorem{definition}[theorem]{Definition}
\newtheorem{lemma}[theorem]{Lemma}
\newtheorem{corollary}[theorem]{Corollary}

\journal{Games and Economic Behavior}

\begin{document}

\begin{frontmatter}



\title{Regular Potential Games}

 \author[label3]{Brian Swenson}
 \author[label2]{Ryan Murray}
 \author[label1]{Soummya Kar}
 \address[label3]{Department of Electrical Engineering, Princeton University, Princeton, NJ (bswenson@princeton.edu)}
 \address[label2]{Department of Mathematics, North Carolina State University, Raleigh, NC (rwmurray@ncsu.edu)}
 \address[label1]{Department of Electrical and Computer Engineering, Carnegie Mellon University, Pittsburgh, PA (soummyak@andrew.cmu.edu)}

%\author{Brian Swenson, Ryan Murray, and Soummya Kar}

%\address{}

\renewcommand{\thefootnote}{\fnsymbol{footnote}}

%\footnotetext[2]{Department of Electrical and Computer Engineering,
%Carnegie Mellon University, Pittsburgh, PA, USA (brianswe@ece.cmu.edu, soummyak@ece.cmu.edu).}
%\footnotetext[3]{Department of Mathematics, Pennsylvania State University, State College, PA, USA (rwm22@psu.edu).}

\renewcommand{\thefootnote}{\arabic{footnote}}

\begin{abstract}
A fundamental problem with the Nash equilibrium concept is the existence of certain ``structurally deficient'' equilibria that (i) lack fundamental robustness properties, and (ii) are difficult to analyze.
The notion of a ``regular'' Nash equilibrium was introduced by Harsanyi.
Such equilibria are isolated, highly robust, and relatively simple to analyze. A game is said to be regular if all equilibria in the game are regular.
In this paper it is shown that almost all potential games are regular. That is, except for a closed subset with Lebesgue measure zero, all potential games are regular.
As an immediate consequence of this, the paper also proves an oddness result for potential games:
in almost all potential games, the number of Nash equilibrium strategies is finite and odd. Specialized results are given for weighted potential games, exact potential games, and games with identical payoffs. Applications of the results to
game-theoretic learning are discussed.
\end{abstract}


\end{frontmatter}


\section{Introduction} \label{sec_intro}

While the notion of Nash equilibrium (NE) is a universally accepted solution concept for games, several shortcomings have been noted over the years.
A principal criticism (in addition to non-uniqueness) is
that some Nash equilibrium strategies may be undesirable or unreasonable due to a lack of basic robustness properties.
As a consequence, many equilibrium refinement concepts have been proposed \cite{selten1975reexamination,myerson1978refinements,harsanyi1973oddness,van1991stability,wen1962essential,kojima1985strongly,okada1981stability},
each attempting to single out subsets of Nash equilibrium strategies that satisfy some desirable criteria.

One of the most stringent refinement concepts, originally proposed by Harsanyi \cite{harsanyi1973oddness}, is that a NE strategy be ``regular.''
%The notions of regular and quasi-strong equilibria were originally proposed by Harsanyi \cite{harsanyi1973oddness}.
In the words of van Damme \cite{van1991stability}, ``regular Nash equilibria possess all the robustness properties that one can reasonably expect equilibria to possess.'' Such equilibria are quasi-strict \cite{harsanyi1973oddness,van1991stability},  perfect \cite{selten1975reexamination}, proper \cite{myerson1978refinements}, strongly stable \cite{kojima1985strongly}, essential \cite{wen1962essential}, and isolated \cite{van1991stability}.\footnote{See \cite{van1991stability} for an in-depth discussion of each of these concepts and their interrelationships.}

If all equilibria of a game are regular, then the number of NE strategies in the game has been shown to be finite and, curiously, odd \cite{harsanyi1973oddness,wilson1971computing}.
Regular equilibria have also been studied in the context of games of incomplete information, where, as part of Harsanyi's celebrated purification theorem \cite{harsanyi1973games,govindan2003short,morris2008purification}, they have been shown to be approachable.

A game is said to be regular if all equilibria in the game are regular. Harsanyi \cite{harsanyi1973oddness} showed that almost all\footnote{Following Harsanyi \cite{harsanyi1973oddness}, when we say almost all games satisfy some condition we mean the set of games where the condition fails to hold is a closed set with Lebesgue measure zero. See Section~\ref{sec_almost_all} for more details.} games are regular, and hence, in almost all games, all equilibria possess all the robustness properties we might reasonably hope for.

While this result is a powerful when targeted at general $N$-player games, there are many important classes of games that have Lebesgue measure zero within the space of all games \cite{brandt2008hardness}. Harsanyi's result tells us nothing
%is not sufficiently fine to tell us anything
about equilibrium properties within such special classes of games.
This is the case, for example, in the important class of multi-agent games known as \emph{potential games} \cite{Mond96}.

A game is said to be a \emph{potential game} if there exists some underlying function (generally referred to as the \emph{potential function}) that all players implicitly seek to optimize.
Potential games have many applications in economics and engineering \cite{marden2009cooperative,Mond96}, and are particularly useful in the study of multi-agent systems, e.g., \cite{scutari2006potential,xu2013decision,zhu2013distributed,ding2012collaborative,
nie2006adaptive,chu2012cooperative,srivastava2005using,Lambert01,
garcia2000fictitious,marden-connections,li2013designing}.

There are several types of potential games---in order of decreasing generality, these include weighted potential games, exact potential games, and games with identical payoffs \cite{Mond96,Mond01}.
%\footnote{More general sets of potential games include ordinal potential games \cite{Mond96} and best response potential games \cite{Voor00}. In this paper we will focus on weighted potential games and subsets thereof.}
Letting WPG, EPG, and GIP denote the set of each of these types of potential games respectively, and letting G denote the set of all games, we have the following relationship:\footnote{More precisely, any finite game of a fixed size (i.e., with a fixed number of players and actions) is uniquely represented as a vector in Euclidean space denoting the payoff received by each player for each pure strategy. The set $G$ of all possible games of a given size is equal to $\R^m$ for some appropriate $m\in\N$ (see, e.g., \cite{fudenberg1998theory} Section 12.1), and each class of potential games is a lower-dimensional subset of $\R^m$. See Section~\ref{sec_almost_all} for more details.}
%among these sets of games
$$
\mbox{GIP} \subset \mbox{EPG} \subset \mbox{WPG} \subset \mbox{G},
$$
where each subset is a low-dimensional (measure-zero) subset within any of its supersets.
Harsanyi's regularity result provides no information on the abundance (or dearth) of regular equilibria within these subclasses of games. Hence, when restricting attention to potential games, as is often done in the study of multi-agent systems, we are deprived of any generic results on the regularity, robustness, or finiteness of the equilibrium set.

We say that a property holds for almost all games in a given class if the subset of games in the class where the property fails to holds is a closed set with Lebesgue measure zero (with the dimension of the Lebesgue measure corresponding to the dimension of the given class of games---see Section~\ref{sec_almost_all} for more details).
%We say a property holds for almost all games of a given class if the set of games where the property fails to holds has (appropriately-dimensioned) Lebesgue measure zero.

The main result of this paper is the following theorem.
\begin{theorem} \label{thrm_regular_eqilibria} %\label{thrm_regular_eqilibria}
$~$\\
(i) Almost all weighted potential games are regular.\\
(ii) Almost all exact potential games are regular.\\
(iii) Almost all games with identical payoffs are regular.
\end{theorem}
We note that this result implies that for almost all games in each of these classes, all equilibria are quasi-strict, perfect, proper, strongly stable, essential, and isolated.
Using Harsanyi's oddness theorem (see \cite{harsanyi1973oddness}, Theorem 1), we see that in any regular game, the number of NE strategies is finite and odd. Hence, the following result is an immediate consequence of Theorem~\ref{thrm_regular_eqilibria}.
\begin{theorem}
In almost all weighted potential games, almost all exact potential games, and almost all games with identical payoffs, the number of NE strategies is finite and odd.
\end{theorem}

Regularity may be seen as serving two purposes. First, it ensures that the equilibrium set possesses the desirable structural properties noted above (e.g., equilibria are isolated, robust, and finite in number). Second, it simplifies the analysis of the game near equilibrium points---the important features of players' utility functions near an equilibrium can be understood by looking only at first- and second-order terms in the associated Taylor series expansion.
In this sense, the role of regular equilibria in games is analogous to the role that non-degenerate critical points play in the study of real-valued functions.\footnote{A critical point $x^*$ of a function $f:\R^n\to \R$ is said to be non-degenerate if the Hessian of $f$ at $x^*$ is non-singular. When a critical point is non-degenerate, one can understand the important local properties of $f$ using only the gradient and Hessian of $f$. If a critical point is \emph{degenerate} then heavy algebraic machinery may be required to understand the local properties of $f$.
With regard to games, if $x^*$ is an interior equilibrium point of a potential game with potential function $U$, then $x^*$ is regular if and only if $x^*$ is a non-degenerate critical point of $f$. For non-interior equilibrium points the story is more involved, but the main idea is the same.}
This amenable analytic structure can greatly facilitate the study of (for example) game-theoretic learning processes \cite{swens2017BRdynamics,cohen2017learning}
%\cite{roth2013stochastic}
or approachability in games with incomplete information \cite{govindan2003short}.

As an application of these results to learning theory, in the paper \cite{swens2017BRdynamics} we consider the problem of studying continuous best-response dynamics (BR dynamics) \cite{gilboa1991social,hofbauer1995stability,hofbauer2003evolutionary} in potential games. BR dynamics are fundamental to learning theory---they model various forms of learning in games and underlie many popular game-theoretic learning algorithms including the canonical fictitious play (FP) algorithm \cite{benaim2005stochastic}. While it is known that BR dynamics converge to the set of NE in potential games, the result is less than satisfactory. BR dynamics can converge to mixed-strategy (saddle-point) Nash equilibria and solutions of BR dynamics may be non-unique.
Furthermore, little is understood about transient properties such as the rate of convergence of BR dynamics in potential games. (In fact, due to the non-uniqueness of solutions in potential games, it has been shown that it is impossible to establish convergence rate estimates for BR dynamics that hold at all points \cite{harris1998rate}.)

In \cite{swens2017BRdynamics} we study how regular potential games can be used to address these issues. In particular, it is shown that in any regular potential game, BR dynamics converge generically to pure NE, solutions of BR dynamics are generically unique, and the rate of convergence of BR dynamics is generically exponential. Combined with the results of the present paper, this allows us to show that BR dynamics are ``well behaved'' in almost all potential games.

Furthermore, in \cite{Mond01} Monderer and Shapley study the convergence of the closely related FP algorithm in potential games and show convergence to the set of NE. In particular, they show that it is possible for FP to converge to completely mixed NE in potential games, which can be highly problematic for a number of reasons \cite{jordan1993,Fud92,swens2017BRdynamics}.
However, they conjecture that such behavior is exceptional;
that is, they conjecture that in generic two-player potential games FP always converges to pure NE (see \cite{Mond01}, Section 2).\footnote{Since their reasoning relies on the improvement principle \cite{monderer1997fictitious}, which does not hold in games with more than two players, they limit their conjecture to two-player games.}
Regular potential games are well suited to studying this conjecture.\footnote{Harsanyi's result (that regular games are generic in the space of all games) does not aid in addressing this conjecture which deals specifically with regular potential games.}
Theorem 1 of the present paper combined with Theorem 1 of \cite{swens2017BRdynamics} shows that for the continuous-time version of FP \cite{harris1998rate,shamma2004unified} (which is equivalent to BR dynamics after a time change \cite{harris1998rate}), this conjecture holds generically for potential games of arbitrary size; that is, in any regular potential game (and hence almost all potential games) continuous-time FP dynamics converge to pure NE from almost all initial conditions.

While regular equilibria have traditionally been studied as an equilibrium selection criterion in games, in potential games, pure strategy NE are naturally selected by virtue of the potential function. We emphasize that our objective in studying regular potential games is not to argue that (possibly mixed) regular equilibria are more natural than pure equilibria in potential games. Rather, we wish to demonstrate that (i) in almost all potential games, all equilibria possess desirable structural properties, (ii) properties that hold for regular potential games are inherently robust to payoff perturbations, and (iii) properties that hold only for irregular potential games are inherently nonrobust to payoff perturbations. We note, however, that our main result does imply that the two considerations (having a pure strategy regular NE) can be aligned in generic (i.e., regular) potential games where a potential maximizer must also be a regular equilibrium. We also remark that, as discussed earlier, regular potential games can be particularly useful in the study of game-theoretic learning processes, and a key aim of the present work is to facilitate the study of such learning processes by establishing the genericity of regular potential games.

The remainder of the paper is organized as follows.
In Section~\ref{sec_degen_games_pf_strategy} we outline our strategy for proving Theorem~\ref{thrm_regular_eqilibria}.
Section~\ref{sec_prelims} sets up notation.
%Section \ref{sec_almost_all} defines the notion of ``almost all games'' used in the paper.
Section~\ref{sec_reg_eq} presents a pair of key non-degeneracy conditions that are equivalent to regularity in potential games.
Section~\ref{sec_nondegenerate_games} presents Proposition~\ref{lemma_second_order_degeneracy}  which states that almost all identical payoff games are second-order non-degenerate (this proposition is the technical core of the paper) and sets up notation for analyzing second-order degeneracy (and regularity). Section~\ref{sec:general-games} takes a brief digression to contrast proof techniques in general games vs potential games. Section~\ref{sec:proof-2nd-order-degen} proves Proposition~\ref{lemma_second_order_degeneracy} .
Section~\ref{sec_first_order_degenerate_games} proves that almost all identical payoff games are first-order non-degenerate.
Section~\ref{sec_exact_and_weighted_games} proves that almost all exact and weighted potential games are regular. Section~\ref{sec_conclusion} concludes the paper.

\subsection{Proof Strategy} \label{sec_degen_games_pf_strategy}
Our strategy for proving Theorem~\ref{thrm_regular_eqilibria} is as follows.
A potential game will be seen to be regular if and only if the corresponding identical payoffs game (with the potential function being the common utility) is regular.  Thus, the problem of proving Theorem~\ref{thrm_regular_eqilibria} reduces, by and large, to the following proposition.
\begin{proposition} \label{prop_identical_payoff}
Almost all games with identical payoffs are regular.
\end{proposition}
The bulk of the paper will be devoted to proving this proposition. We will prove it as follows.

An equilibrium in an identical-payoffs (or potential) game can be shown to be regular if and only if the derivatives of the potential function satisfy two simple non-degeneracy conditions which we refer to as first- and second-order non-degeneracy (see Section~\ref{sec_reg_equilib_pot_games}).
%(see Section~\ref{sec_reg_eq} and Lemma \ref{lemma_non_degen_to_regular}).
The first-order condition deals with the gradient of the potential function. (A first-order non-degenerate equilibrium is referred to as a \emph{quasi-strong} equilibrium in \cite{harsanyi1973oddness}, and a \emph{quasi-strict} equilibrium in other works \cite{van1991stability}.)\footnote{We prefer the term first-order non-degenerate name because it emphasizes the role of the potential function and is consistent with the notion of second-order degeneracy.} An equilibrium $x^*$ is second-order non-degenerate if the Hessian of the potential function taken with respect to the support of $x^*$ is invertible.

We say a potential game is first-order (second-order) non-degenerate if all equilibria in the game satisfy the first-order (second-order) non-degeneracy conditions.
We will prove Proposition~\ref{prop_identical_payoff} by showing that:
\begin{enumerate}[label=(\roman*)]
\item  An equilibrium of an identical payoffs game is regular if and only if it is first- and second order non-degenerate (see Section~\ref{sec_reg_eq} Lemma~\ref{lemma_non_degen_to_regular}),
\item Almost all games with identical payoffs are second-order non-degenerate (see Section~\ref{sec_second_order_degenerate_games} Proposition~\ref{lemma_second_order_degeneracy}),
\item Almost all games with identical payoffs are first-order non-degenerate (see Section~\ref{sec_first_order_degenerate_games} Proposition~\ref{lemma_first_order_degeneracy}).
\end{enumerate}
We remark that in Propositions~\ref{lemma_second_order_degeneracy} and \ref{lemma_first_order_degeneracy} we show that the subset of irregular games has (appropriately dimensioned) Lebesgue measure zero. By Remark~\ref{remark_closedness} this is sufficient to establish regularity in almost all identical payoff games.

After proving Proposition~\ref{prop_identical_payoff}, we extend the result to exact and weighted potential games. This extension is straightforward and is accomplished in Section~\ref{sec_exact_and_weighted_games} (see Proposition~\ref{prop_exact_and_weighted}). Propositions~\ref{prop_identical_payoff} and~\ref{prop_exact_and_weighted} together prove Theorem~\ref{thrm_regular_eqilibria}.

\section{Notation} \label{sec_prelims}
We will outline the notation and terminology used throughout the paper below. A short glossary of additional standard notation used in the paper can be found in \ref{Appendix-math-notation}.

\subsection{Normal Form Games} \label{sec:prelim-1}
A game in normal form is given by a tuple
$\Gamma:= (N,(Y_i,u_i)_{i=1,\ldots,N}),$ where $N\in\{1,2,\ldots\}$ denotes the number of players, $Y_i:=\{y_i^1,\ldots,y_i^{K_i}\}$ denotes the set of pure strategies (or actions) available to player $i$, with cardinality $K_i := |Y_i|$, and $u_i:\prod_{j=1}^N Y_j \rightarrow \mbb{R}$ denotes the utility function of player $i$.

Given some game $\Gamma$, let $Y:= \prod_{i=1}^N Y_i$ denote the set of joint pure strategies available to players, and let
\begin{equation} \label{def_K}
K := K_1\times\cdots\times K_N
\end{equation}
denote the number of joint pure strategies.
When defining spaces of games (e.g., as in Section~\ref{sec_almost_all}) we will find it convenient to view $u _i= (u_i(y)_{y\in Y})$ as a vector in $\R^K$; we will clearly indicate when using this abuse of notation.

The set of mixed strategies of player $i$ is typically defined to be the probability simplex over $A_i$. It will simplify the presentation to consider an equivalent, but slightly modified, definition of the set of mixed strategies. Let
\begin{equation}\label{def_X-space}
X_i := \{x_i\in \R^{K_i-1}:~ 0\leq x_i^k\leq 1 \mbox{ for } k=1,\ldots,K_i-1, \mbox{ and } \sum_{k=1}^{K_i-1}x_i^k \leq 1\},
\end{equation}
denote the set of mixed strategies of player $i$. (This definition will allow us to perform calculations without being directly encumbered by the hyperplane constraint inherent in the probability simplex.) A strategy $x_i \in X_i$ is interpreted as follows: The scalar $1-\sum_{k=1}^{K_i} x_i^k$ represents the weight placed on the first pure strategy $y_i^1\in Y_i$, and for $k\in \{2,\ldots,K_i-1\}$, $x_i^{k-1}$ represents the weight placed on the $k$-th pure strategy $y_i^{k}$. In order to later reference this interpretation, the following notation will be useful. Given $x_i\in X_i$, let
\begin{equation}\label{def_T}
T_i^1(x_i) := 1-\sum_{k=1}^{K_i-1} x_i^k \quad \text{ and } \quad T_i^k(x_i) := x_i^{k-1}, ~~ k\geq 2.
\end{equation}

Let $X := \prod_{i=1}^N X_i$ denote the set of joint mixed strategies and let $X_{-i}:=\prod_{j\not = i} X_j$. When convenient, given a mixed strategy $x=(x_1,\ldots,x_N)\in X$, we use the notation $x_{-i}$ to denote the tuple $(x_j)_{j\not=i}$.

Given a mixed strategy $x\in X$, the expected utility of player $i$ is given by
\begin{equation} \label{eq_potential_expanded_form2}
U_i(x) = \sum_{k=1}^{K_i-1}x_i^k U_i(y_i^{k+1},x_{-i}) + \left(1-\sum_{k=1}^{K_i-1}x_i^k\right)U_i(y_i^{1},x_{-i}).
\end{equation}

A strategy $x\in X$ is said to be a \emph{Nash equilibrium} (or simply an \emph{equilibrium}) if $x_i \in \arg\max_{x_i'\in X_i} U_i(x_i',x_{-i})$ for all $x_i'\in X_i$, $i=1,\ldots,N$.

We say that a strategy $x_i$ is completely mixed if it lies in the interior of $X_i$ and we say that $x_i$ is a pure strategy if is a vertex of $X_i$. Otherwise, we say that $x_i$ is an incompletely mixed strategy.

%We now define the notion of the \emph{carrier set} of an element $x\in X$, a natural modification of a support set to the present context.
For $x_i \in X_i$, we define the \emph{carrier set} of $x_i$ to be
%(or $x_i\in X_i$)
\begin{equation} \label{def_carrier}
\carr_i(x_i) := \{y_i^k\in Y_i: T_i^k(x_i)>0\},
%x_i^{k-1} > 0, ~k\geq 2\}.
\end{equation}
that is, $\carr_i(x_i)$ is the subset of pure strategies in $Y_i$ that receive positive weight under $x_i$. (This is just a slight modification of the notion of a \emph{support set} to the present context.)
For $x = (x_1,\ldots,x_N)\in X$ let $\carr(x) := \carr_1(x_1)\cup\cdots\cup\carr_N(x_N)$.

Suppose $C=C_1\cup\cdots\cup C_N$, where for each $i=1,\ldots,N$, $C_i$ is a nonempty subset of $Y_i$. We say that $C$ is the carrier for $x=(x_1,\ldots,x_N)\in X$ if $C=\carr(x)$, or equivalently, if $C_i = \carr_i(x_i)$ for $i=1,\ldots,N$. 

\subsection{Potential Games} \label{sec_pot_games_prelim}
Following \cite{Mond96,Mond01}, a game is said to be a \emph{weighted potential game} if there exists a function
$u:Y\rightarrow \R$ and a vector of positive weights $(w_i)_{i=1}^N\in\R^N$ such that
\begin{equation} \label{eq_def_weighted_pot}
u_i(y_i',y_{-i}) - u_i(y_i'',y_{-i}) = w_i\big(u(y_i',y_{-i}) - u(y_i'',y_{-i})\big)
\end{equation}
A game is said to be an \emph{exact potential game} if \eqref{eq_def_weighted_pot} holds with $w_i=1$ for all $i=1,\ldots,N$. A game is said to be an \emph{identical-payoffs game} if there exists a function $u:Y\rightarrow \R$ such that $u_i(y) = u(y)$ for all $y\in Y$, $i=1,\ldots,N$.

When we refer simply to a ``potential game'' we mean a weighted potential game, which includes the other classes of games as special cases.

Given a potential game, we let $U: X\to \R$ be the multilinear extension of $u$ defined by
\begin{equation} \label{eq_potential_expanded_form2}
U(x) = \sum_{k=1}^{K_i-1}x_i^k U(y_i^{k+1},x_{-i}) + \left(1-\sum_{k=1}^{K_i-1}x_i^k\right)U(y_i^{1},x_{-i}).
\end{equation}
We will typically refer to $U$ as the \emph{potential function} and to $u$ as the \emph{pure form of the potential function}.

By way of notation, given a pure strategy $y_i \in Y_i$ and a mixed strategy $x_{-i} \in X_{-i}$, we will write $U(y_i,x_{-i})$ to indicate the value of $U$ when player $i$ uses a mixed strategy placing all weight on the $y_i$ and the remaining players use the strategy $x_{-i}\in X_{-i}$.

\subsection{Almost All Games} \label{sec_almost_all}

The following definition specifies the notion of ``almost all'' to be used in the paper.
\begin{definition} \label{def_almost_all}
Given some property $\calS$, we say that $\calS$ holds for almost all points in some Euclidean space $\R^n$ if the set where $\calS$ fails to hold is a closed set with $\calL^n$-measure zero.
\end{definition}

We say that a game $\Gamma$ has size $(N,(K_i)_{i=1}^N)$ if $\Gamma$ is a general $N$-player game and the size of the action space of each player satisfies $|Y_i|= K_i \in \{2,3,\ldots\}$, $i=1,\ldots,N$. Let $K$ be as defined in \eqref{def_K} so that $K$ gives the number of pure strategies in a game of size $(N,(K_i)_{i=1}^N)$.

%Suppose that a game of size $(N,(K_i)_{i=1}^N)$ is fixed.
A game of size $(N,(K_i)_{i=1}^N)$ is uniquely represented by the vector $u:=(u_i(y))_{y\in Y,~ i=1,\ldots,N} \in \R^{NK}$ which specifies the utility received by each player for each pure strategy $y\in Y$. We will frequently refer to $u$ as the vector of \emph{utility coefficients}. The set of all games of this size is equivalent to $\R^{NK}$.
%We say that almost all games satisfy a given property if, for any game size  $(N,(K_i)_{i=1}^N)$, the set of games

Let $\calW, \calP, \calI\subset \R^{NK}$, denote the subsets of weighted potential games, exact potential games, and identical-payoffs games respectively. (An explicit construction of $\calW$ and $\calP$ can be found in Section~\ref{sec_exact_and_weighted_games}.) Let $K_w := \dim \calW$ and $K_p:= \dim\calP$. The sets of weighted and exact potential games are equivalent to the Euclidean spaces $\R^{K_w}$ and $\R^{K_p}$ respectively.

%For the set of identical payoff games, we have $\dim \calI = K$, where $K$ is given in \eqref{def_K}. In particular,
An identical-payoffs game is uniquely represented as a vector $u\in \R^K$ denoting the payoff (identical for all players) received for each action $y\in Y$, and we will represent this set of games as  $\calI := \R^K.$

%Per Definition \ref{def_almost_all},
For a fixed game size, each class of games discussed above is equivalent to some Euclidean space.
%For a set game size, the each of the above game classes corresponds to Euclidean space of a certain dimension.
We will say that almost all games in a given class are regular if, for any game size, almost all games (per Definition~\ref{def_almost_all}) in the class are regular.% where the dimension of the Lebesgue measure corresponds to the dimension of the corresponding Euclidean space.

\begin{remark}[Closedness] \label{remark_closedness}
It was shown by Harsanyi \cite{harsanyi1973oddness,van1991stability} that the set of irregular  games of (any) size $(N,(K_i)_{i=1}^N)$ is a closed subset in the space of all games of size $(N,(K_i)_{i=1}^N)$. It is straightforward to extend this result to the various subclasses of potential games. In particular, we have:\\
(i) The set of irregular weighted potential games is closed with respect to $\calW$.\\
(ii) The set of irregular exact potential games is closed with respect to $\calP$.\\
(iii) The set of irregular games with identical payoffs is closed with respect to $\calI$.

For practical purposes this means that in order to verify that almost all potential games of a given class are regular, we need only verify that the subset of irregular games has appropriately dimensioned Lebesgue measure zero.
\end{remark}

\section{Regular Equilibria in Potential Games} \label{sec_reg_eq} \label{sec_reg_equilib_pot_games}
In potential games, regular equilibria have a simple and intuitive meaning.\footnote{In general games, the definition of a regular equilibrium is somewhat abstruse, relying on the invertibility of the Jacobian of a particular map \cite{van1991stability}. See Section~\ref{sec:def-regularity-standard} for more details.} An equilibrium in a potential game is regular if (i) it is quasi-strict in the sense of \cite{van1991stability} (a condition we will refer to as first-order non-degeneracy), and (ii) the Hessian of $U$ (the potential) taken with respect to the support of the equilibrium is invertible (a condition we will refer to as second order non-degeneracy). For equilibria in the interior of the strategy space, this simply reduces to an equilibrium being regular if and only if it is a non-degenerate critical point of $U$ in the traditional sense.

%The traditional definition of regularity will be given in Section ?? when we discuss general games.

Since the non-degeneracy conditions given in this section (applicable only within the class of potential games) are simpler to work with than the traditional definition of regularity, we will work with directly with these conditions through the remainder of the paper. However, the traditional definition of regularity can be found in Section~\ref{sec:general-games} (see Definition~\ref{def_regular_equilib}) where we contrast our techniques with those required in general games.

The remainder of the section is organized as follows.
In Sections~\ref{sec_first_order_degeneracy}--\ref{sec_second_order_degeneracy} we define the notions of first and second-order degeneracy. In Section~\ref{sec_degen_and_reg_equilibria} we show that, in a potential game, these conditions are equivalent to regularity.

\subsection{First-Order Degeneracy} \label{sec_first_order_degeneracy}
%Let $\Gamma$ be a game with identical payoffs.
Let $C=C_1\cup\cdots\cup C_N$, $C_i\subset Y_i$ $\forall i=1,\ldots,N$ be some carrier set.
Let $\gamma_i := |C_i|$
and assume that the strategy set $Y_i$ is reordered so that $\{y_i^1,\ldots,y_i^{\gamma_i}\}=C_i$. Under this ordering, the first $\gamma_i-1$ components of any strategy $x_i$ with $\carr_i(x_i) = C_i$ are free (not constrained to zero by $C_i$) and the remaining components of $x_i$ are constrained to zero. That is, the subvector $(x_i^k)_{k=1}^{\gamma_i-1}$ is free under $C_i$ and the subvector $(x_i^k)_{k=\gamma_i}^{K_i}=0$.
The set of strategies $\{x\in X:~\carr(x) = C\}$ is precisely the interior of the face of $X$ given by
\begin{equation}\label{def_C_face}
X_C :=\{x\in X:~ x_i^k = 0,~ k=\gamma_i,\ldots,K_i-1,~i=1,\ldots,N \}.
\end{equation}

\begin{definition} [First-Order Degenerate Equilibrium] \label{def_first_order_degen}
Suppose $\Gamma$ is a potential game with potential function $U$.
Suppose $x^*\in X$ is an equilibrium of $\Gamma$ with carrier $C$. We say that $x^*$ is a \emph{first-order degenerate} equilibrium if there exists a pair $(i,k)$, $i=1,\ldots,N$, $k=\gamma_i,\ldots,K_i-1$ such that
\begin{equation} \label{eq_first_order_condition}
\frac{\partial U(x^*)}{\partial x_i^k}=0
\end{equation}
and we say $x^*$ is \emph{first-order non-degenerate} otherwise.
\end{definition}
We note that the condition \eqref{eq_first_order_condition} in the definition of first-order degeneracy is implicitly dependent on the carrier $C$ (since $\gamma_i = |C_i|$ and we assume that $C_i= \{y_i^1,\ldots,y_i^{\gamma_i}\}$).
\begin{definition}[First-order Degenerate Game]
We say a game is first-order degenerate if it has an equilibrium that is first-order degenerate, and we say the game is first-order non-degenerate otherwise.
\end{definition}
\begin{example}
Consider the $3\times 2$ identical-payoffs game with payoff matrices
%$$
%M_1 =
%\begin{pmatrix}
%  -4 & 1 \\
%  1 & -1
%\end{pmatrix}
%\quad
%M_2 =
%\begin{pmatrix}
%  1 & -1 \\
%  2 & 1
%\end{pmatrix},
%$$
$$
M_1 =
\begin{pmatrix}
  1 & -4 \\
  -1 & 1
\end{pmatrix}
\quad
M_2 =
\begin{pmatrix}
  -1 & 1 \\
  1 & 2
\end{pmatrix},
$$
where player 1 is the ``row player'', player 2 is the ``column player'',  and player 3 is the ``matrix player.'' Recalling that $x_i$ denotes the probability of player $i$ playing action $y_i^2$ (since each player has only two strategies, we drop the superscipts on $x_i$),
this game has a first-order degenerate equilibrium at the strategy $x^*=(1/2,0,1/2)$.

This may be visualized in terms of the gradient of $U$. The three (gradient) level surfaces along which $\frac{\partial U(x)}{\partial x_i} = 0$, $i=1,2,3$, are displayed in Figures~\ref{fig:view1}--\ref{fig:view2}. The equilibrium $x^*$ has carrier $C = Y_1 \cup \{y_2^1\} \cup Y_3$ and lies on the face of $X$ given by $\{x\in X: x_2=0\}$. Yet, $x^*$ also lies at the intersection of all three level surfaces (and thus is a bona fide critical point of $U$). In particular, $x^*$ lies on the level surface, $\{x: \frac{\partial U(x^*)}{\partial x_2}=0\}$, making it first-order degenerate.
\begin{figure}[h]
    \centering
    \begin{subfigure}{.45\textwidth}
        %\centering
        \includegraphics[height=.85\textwidth]{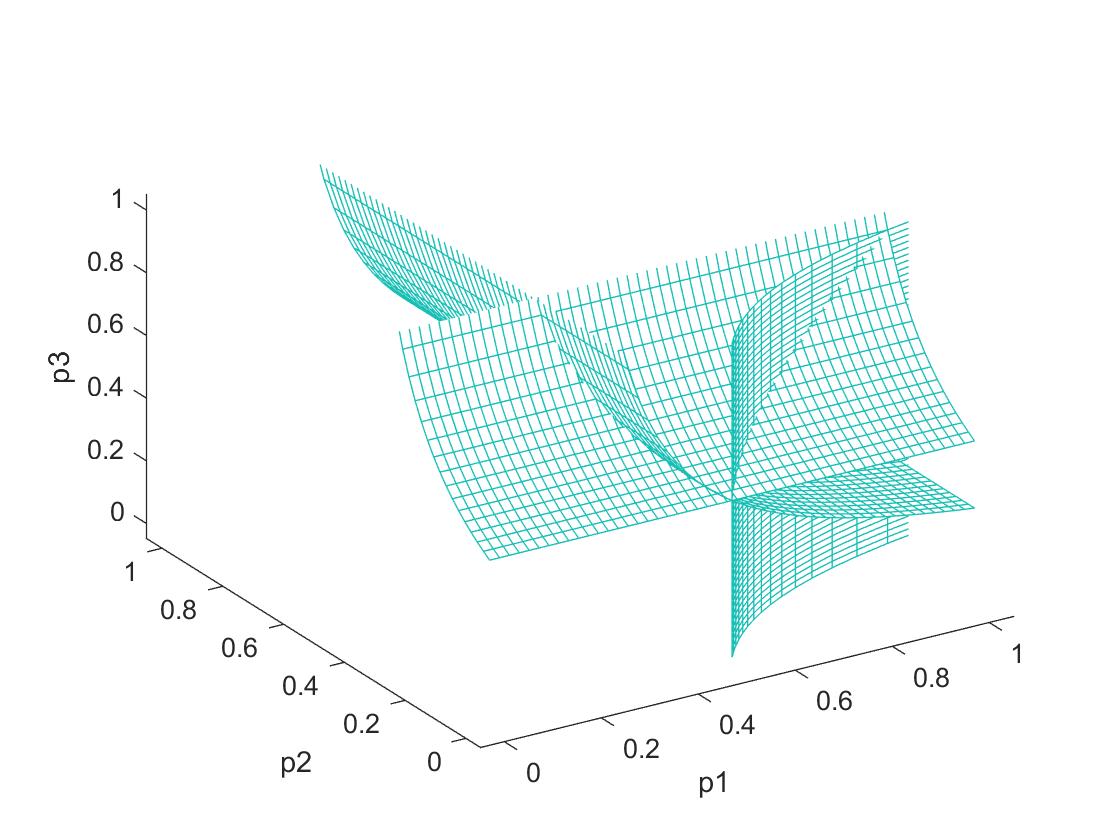}
        \caption{Level surfaces of $\frac{\partial U(x)}{\partial x_i}$, $i=1,2,3$.}
        \label{fig:view1}
    \end{subfigure}
    \hspace{.8cm}
    \begin{subfigure}{.45\textwidth}
        %\centering
        \includegraphics[height=.85\textwidth]{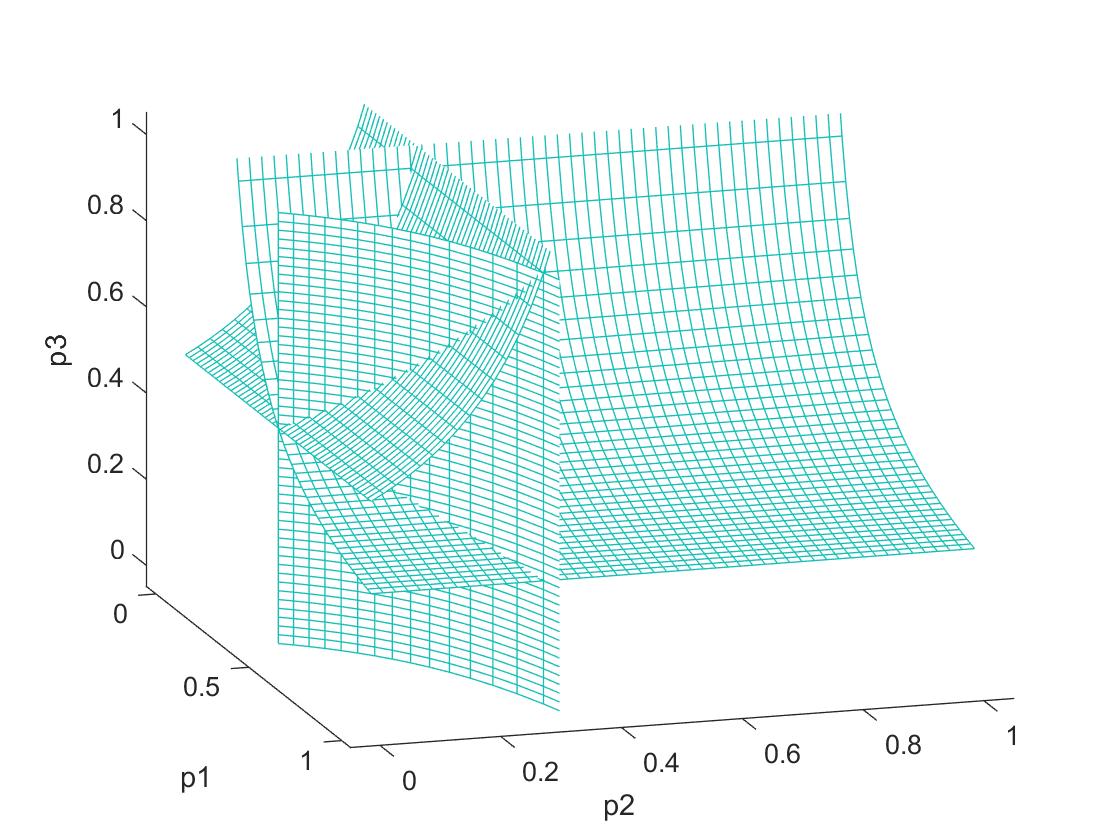}
        \caption{Alternate view of gradient level surfaces.}
        \label{fig:view2}
    \end{subfigure}
\end{figure}
\end{example}

\begin{remark} [Equivalence to Quasi-Strict Equilibria] \label{remark_QS_first-order}
An equilibrium $x^*$ with carrier $C$ is first order non-degenerate if and only if, for every player $i$, the set of pure-strategy best responses to $x_{-i}^*$ coincides with $C_i$.
(This is verified using the multi-linearity of $U$.)
We note that, using this later definition, Harsanyi \cite{harsanyi1973oddness} referred to first-order non-degenerate equilibria as \emph{quasi-strong} equilibria. In other works these have been referred to as \emph{quasi-strict} equilibria \cite{van1991stability}.
%it is quasi-strong, as introduced by Harsanyi \cite{harsanyi1973oddness}.
We prefer to use the term first-order non-degenerate in order to emphasize that we are concerned with the gradient of the potential function and to keep the nomenclature consistent with the notion of second-order non-degeneracy, introduced next.
\end{remark}
%\begin{remark}
%In Definition \ref{def_first_order_degen} we considered an arbitrary potential function $U$ associated with $\Gamma$. This is justified since \eqref{eq_first_order_condition} holds for some potential function $U$ of $\Gamma$ if and only if it holds for every potential function of $\Gamma$. This follows readily by differentiating \eqref{eq_potential_expanded_form2} and using Definition \ref{def_weighted_pot_game} and the definitions of expected utility/potential \eqref{def_U_i},\eqref{def_potential_fun1}.
%\end{remark}

\subsection{Second-Order Degeneracy} \label{sec_second_order_degeneracy}
Suppose that $\Gamma$ is a potential game with potential function $U$.
Let $C$ be some carrier set. Let $\tilde{N}:=|\{i=1,\ldots,N:~\gamma_i \geq 2\}|$, and assume that the player set is ordered so that $\gamma_i\geq 2$ for $i=1,\ldots,\tilde{N}$. Under this ordering, for strategies with $\carr(x) = C$, the first $\tilde{N}$ players use mixed strategies and the remaining players use pure strategies.
Assume that $\tilde{N}\geq 1$ so that any $x$ with carrier $C$ is a genuine mixed (not pure) strategy.
% so that an $x$ with carrier $C$ is a mixed strategy.

Let the Hessian of $U$ taken with respect to $C$ be given by
\begin{equation} \label{def_mixed_hessian}
\tilde{\vH}(x) :=\left( \frac{\partial^2 U(x)}{\partial x_i^k \partial x_j^\ell} \right)_{\substack{i,j=1,\ldots,\tilde{N},\\ k=1,\ldots,\gamma_i-1,\\ \ell=1,\ldots,\gamma_j-1}}.
\end{equation}
Note that this definition of the Hessian restricts attention to the components of $x$ that are free (i.e., unconstrained) under $C$. That is, $\tilde \vH(x)$ taken with respect to $C$ is the Hessian of $U\vert_{X_C}$ at $x$.

\begin{definition} [Second-Order Degenerate Equilibrium] \label{def_second_order_degen}
Let $\Gamma$ be a potential game with potential function $U$.
We say an equilibrium $x^* \in X$ of $\Gamma$ is \emph{second-order degenerate} if the Hessian
%$\vH_{\carr(x^*)}(x^*)$
$\tilde{\vH}(x^*)$ taken with respect to $\carr(x^*)$
is singular, and we say $x^*$ is \emph{second-order non-degenerate} otherwise.
\end{definition}
\begin{definition}[Second-Order Degenerate Game]
We say a game is second-order degenerate if it has an equilibrium that is second-order degenerate, and we say the game is second-order non-degenerate otherwise.
\end{definition}

\subsection{Regular Equilibria and Non-Degeneracy Conditions} \label{sec_degen_and_reg_equilibria}
The following lemma shows that within the class of potential games, the first and second-order non-degeneracy conditions defined above are equivalent to regularity.
\begin{lemma} \label{lemma_non_degen_to_regular}
Let $\Gamma$ be a potential game.  Then,\\
(i) If an equilibrium $x^*$ is first-order non-degenerate, then it is second-order non-degenerate if and only if it is regular.\\
(ii) If an equilibrium $x^*$ is regular, then it is first-order non-degenerate.

In particular, an equilibrium $x^*$ is regular if and only if it is both first and second-order non-degenerate.
\end{lemma}

The proof of this lemma follows readily from the definitions of regularity (see Definition~\ref{def_regular_equilib}) and first and second-order degeneracy, and is omitted for brevity.
%A proof of the lemma can be found in \ref{appendix_degen_v_regular}.

\section{Second-Order Degeneracy in Games with Identical Payoffs} \label{sec_nondegenerate_games}

%In order to prove Proposition \ref{prop_identical_payoff}, it is sufficient to prove that almost all games are first- and second-order non-degenerate.

%We will now prove Proposition \ref{prop_identical_payoff}, which states that almost all identical interest games are regular.

In the previous section we established that a game is regular if and only if it is first- and second-order non-degenerate. In this section we will state our main result regarding second-order non-degeneracy---namely, that the set of identical payoff games that are second-order degenerate has Lebesgue measure zero (see Proposition~\ref{lemma_second_order_degeneracy}).
Later, in Section~\ref{sec_first_order_degenerate_games} we will show the analogous result for first-order degenerate games (Proposition~\ref{lemma_first_order_degeneracy}).
By Remark~\ref{remark_closedness}, this will be sufficient to prove Proposition~\ref{prop_identical_payoff}.

This section is organized as follows. In Section~\ref{sec_second_order_degenerate_games} we state our main result for second-order degenerate games. In Section~\ref{sec_roadmap} we outline our strategy for proving Proposition~\ref{lemma_second_order_degeneracy}. In Section~\ref{sec_proof-2nd-order} we set up notation required for the analysis of second-order non-degeneracy (and the traditional notion of regularity).
The proof of Proposition~\ref{lemma_second_order_degeneracy} will then be given in Section~\ref{sec:proof-2nd-order-degen} after a brief interlude to discuss general games.

\subsection{Second-Order Degenerate Games} \label{sec_second_order_degenerate_games}

In the following proposition  we will assume the number of players $N$ and action-space sizes $K_i$, $i=1,\ldots,N$ are fixed, and let $K$ be as defined in \eqref{def_K}.
%(cf. \eqref{def_calI}).
%The goal of this subsection is to prove the following proposition.
\begin{proposition} \label{lemma_second_order_degeneracy}
%Almost all potential games are second-order non-degenerate.
The set of identical-payoff games that are second-order degenerate has $\calL^K$-measure zero.
\end{proposition}

This proposition will be proved in Section~\ref{sec_proof-2nd-order}.

\subsection{Roadmap of Proof of Proposition~\ref{lemma_second_order_degeneracy}} \label{sec_roadmap}
Our strategy for proving Proposition~\ref{proposition_A_full_rank} will be as follows.

1. Fix a carrier set $C$ with $|C|=\gamma$. Following \eqref{def_C_face}, let
\begin{equation*} %\label{eq:X-star_C}
\rOmega_C := \{x:\carr(x) = C\}.
\end{equation*}
We will begin by deducing that for any equilibrium $x^*\in\rOmega_C$, the key relation \eqref{eq_A_equals_0} holds, where $\vA(x)$ is defined in \eqref{def_A}.

2. Proposition~\ref{proposition_A_full_rank} shows that for every $x\in \rOmega_C$, the matrix $\vA(x)$ has full row rank.%\footnote{We note that the proof Proposition \ref{proposition_A_full_rank} relies on results from the theory of sign pattern matrices and is somewhat tedious. The reader may wish to skip the proof of this proposition on a first reading.}

3. Using Proposition~\ref{proposition_A_full_rank} we construct a countable cover $(B_\ell)_{\ell\geq 1}$ of $\rOmega_C$ such that for each $\ell$, we may choose $\gamma$ columns of $\vA(x)\in \R^{\gamma\times K}$ so that these same columns are linearly independent for all $x\in B_\ell$. For each $\ell$, the associated columns form an invertible square submatrix of $\vA(x)$ for all $x\in B_\ell$. Using this invertible submatrix and \eqref{eq_A_equals_0}, we construct a mapping $\rho_\ell:B_\ell\times \R^{K-\gamma}\to \R^K$, which allows us to recover the full vector of potential coefficients $u\in\R^K$ given an equilibrium $x^*\in\rOmega_C$ and a $(K-\gamma)$-dimensional subvector of the potential coefficients. (See Lemma~\ref{lemma_cover}, \eqref{eq_rho_def1}--\eqref{eq_rho_def2}, and Corollary~\ref{cor_cover}.)

4. We then demonstrate that if $x^*\in B_\ell$ is a NE of a game with potential-function coefficients $u\in\R^K$, then $x^*$ is second-order degenerate if and only if the Jacobian of $\rho_\ell$ is non-invertible. (See Lemma~\ref{lemma_critical_vals}.) Since the output of $\rho_\ell$ is the full vector of coefficients $u$, this implies that $u$ lies in the set of critical values of $\rho_\ell$ (i.e. the set of outputs of $\rho_\ell$ at which the Jacobian is not invertible).

5. From here, the proof follows from Sard's theorem \cite{hirsch1976differential}. In particular, applying Sard's theorem to $\rho_\ell$ shows that the set of games with second-order degenerate equilibria lying in $B_\ell$ has measure zero. Repeating this argument over all $B_\ell$ (of which there is a countable number) and for all carriers $C$ (of which there is a finite number) yields Proposition~\ref{lemma_second_order_degeneracy}.

We remark that the above proof strategy follows the same general strategy used by Harsanyi in \cite{harsanyi1973oddness} to prove regularity in general games. However, important technical challenges arise in the case of potential games that are not present in general games. This is discussed further in Sections~\ref{sec:gen-games-proof-strategy}--\ref{sec:pot-game-proof-technique}.

\subsection{Analysis of Second-Order Degeneracy: Setup} \label{sec_proof-2nd-order}
%We will prove the proposition using Sard's theorem. Our construction roughly follows that of \cite{harsanyi1973oddness}.
Proposition~\ref{lemma_second_order_degeneracy} will be proved using Sard's theorem as outlined above. In this section we introduce some pertinent notation that will allow us to establish \eqref{eq_A_equals_0}, which is the key relationship allowing us to analyze second-order degeneracy.\footnote{We note that this setup is useful both for analyzing potential games and general games. Indeed, in Section~\ref{sec:general-games} this setup will be reused to establish the key inequality \eqref{eq_B_equals_zero}, which is the analog of \eqref{eq_A_equals_0} in general games.}
%(Indeed, in general games, an analogous result to \eqref{eq_A_equals_0} enables the proof that almost all $N$-player games are regular. This relationship will be explored in Section \ref{sec:general-games}.)

%construct the mapping used in our application of Sard's theorem.

%Before proving the proposition we first introduce some pertinent concepts and notation.

Note that the set of joint pure strategies $Y$ may be expressed as an ordered set $Y=\{y^1,\ldots,y^K\}$ where each $y^\tau\in Y$, is an $N$-tuple of strategies, $\tau\in\{1,\ldots,K\}$. We will assume a particular ordering for this set after Proposition~\ref{proposition_A_full_rank}.

%For each pure strategy $y^{\tau}\in Y$, $\tau=1,\ldots,K$ let $u^{\tau}$ denote the pure-strategy potential associated with playing $y^\tau$; that is, $u^\tau := u(y^\tau)$, where $u$ is the pure form of the potential function defined in Section \ref{sec_prelims}. A vector of \emph{potential coefficients} $u=(u^\tau)_{\tau=1}^K$ is an element of $\R^K$.
%%\footnote{More precisely, $u$ is a vector of \emph{potential function coefficients}. In an abuse of terminology, we generally refer to it simply as a vector of utility coefficients.}

%In an identical interest game, there exists a function $u:Y\to \R$ such that $u_i = u$ for all $i=1,\ldots,N$. We will refer to this common payoff function (both the pure and mixed versions) as the potential function.

In an identical-payoffs game there exists a single function $u:Y\to\R$ such that $u_i = u$ for all $i=1,\ldots,N$. If we consider the vector of utility coefficients $u=(u(y))_{y\in Y}\in R^K$ as a variable, then by \eqref{eq_potential_expanded_form2}, $U$ is linear in $u$.\footnote{The potential function $U$ is, of course, a function of both $x$ and $u$. However, since we will only exploit the dependence on $u$ in this section, we generally stick to the standard game-theoretic convention of writing $U$ as a function of $x$ only \cite{fudenberg1991game}.} At this point we will express $U$ in a more convenient form (see \eqref{def_potential_fun}) which more clearly exposes this relationship.

Let $\tau\in\{1,\ldots,K\}$, $i\in\{1,\ldots,N\}$ and $x_i\in X_i$. We define $q_i^\tau:X_i\rightarrow [0,1]$ by
\begin{equation} \label{def_q_function}
q^\tau_i(x_i) := T_i^k(x_i)
\end{equation}
where $k$ corresponds to the action played by player $i$ in the tuple $y^\tau$, i.e, $(y^\tau)_i = y_i^k$, and where $T_i^k(x_i)$ is defined in \eqref{def_T}. In words, given some mixed strategy $x_i\in X_i$, $q_i^\tau(x_i)$ gives the weight $x_i$ places on the particular action used by player $i$ in the $\tau$-th action tuple $y^\tau$.

In an abuse of notation, given a pure strategy $y_i^k\in Y_i$, we let $q_i^\tau(y_i^k) = 1$ if $(y^\tau)_i = y_i^k$ and $q_i^\tau(y_i^k) = 0$ otherwise.

Given a fixed vector of utility coefficients $u\in\R^K$, the utility function $U:X\rightarrow\R$ may be expressed as
%(see \eqref{def_potential_fun1} and \eqref{def_q_function})
\begin{equation}\label{def_potential_fun}
U(x) = \sum_{\tau=1}^K u^\tau \left[\prod_{i=1}^N q_i^\tau(x_i)\right].
\end{equation}
Note that this form makes it clear that $U$ is linear in $u$.

Now, let $C=C_1\cup\cdots\cup C_N$ be some carrier set. The analysis through the remainder of the section will rely on this carrier set being fixed, and many of the subsequent terms are implicitly dependent on the choice of $C$.
%Suppose $x=(x_1,\ldots,x_N)\in X$ is a strategy with carrier $C^*$.
%Without loss of generality assume that $Y_i$ is ordered so that $y_i^{1}$ is in the carrier $C_i$ of player $i$.
In keeping with our prior convention we let $\gamma_i := |C_i|$, and let $\tilde{N}:=|\{i\in\{1,\ldots,N\}:\gamma_i\geq 2\}|$, and without loss of generality we assume that the strategies in $Y_i$ are ordered so that $C_i = \{y_i^1,\ldots,y_i^{\gamma_i}\}$.
%denote the number of players using mixed strategies under $C$.

Any $x_i$ with carrier $C_i$ has precisely $\gamma_i-1$ free components (i.e., not constrained to zero by $C_i$).
The joint strategy
$x = (x_1,\ldots,x_N)$ is a vector with
\begin{equation} \label{eq_gamma_def}
\gamma := \sum_{i=1}^N (\gamma_i-1)
\end{equation}
free components.

%Recall from \eqref{eq_potential_expanded_form2} that we may express $U(x)$ as
%$$
%U(x) = \sum_{k=1}^{K_i-1}x_i^k U(y_i^{k+1},x_{-i}) + (1-\sum_{k=1}^{K_i-1}x_i^k)U(y_i^{1},x_{-i}).
%$$
By \eqref{eq_potential_expanded_form2} we have that
\begin{equation}\label{def_F}
\frac{\partial U(x)}{\partial x_i^k} = U(y_i^{k+1},x_{-i}) - U(y_i^{1},x_{-i})=:F_i^k(x,u)
\end{equation}
%Let
%\begin{equation}
%F_i^k(x,u) := U(y_i^{k+1},x_{-i}) - U(y_i^1,x_{-i})
%%= \frac{\partial U(x)}{\partial x_i^k}
%\end{equation}
for $i=1,\ldots,\tilde{N},~k=1,\ldots,\gamma_i-1$.
Let
%\footnote{Note that $F_i^k$ and $F$ defined here are identical to \eqref{F_tilde_def}--\eqref{F_tilde_def2}, }
\begin{equation}\label{def2_F}
F(x,u) := \left( F_i^k(x,u) \right)_{\substack{i=1,\ldots,\tilde{N}\\ k=1,\ldots,\gamma_i-1}}.
\end{equation}
%\begin{equation}\label{def2_F}
%F(x,u) := \left(F_1^1(x,u),\ldots,F_1^{\gamma_1-1}(x,u),\ldots,F_i^k(x,u),\ldots,F_{\tilde{N}}^1(x,u),\ldots,F_{\tilde{N}}^{\gamma_{\tilde{N}}-1}(x,u)\right)
%\end{equation}
%Consider now the decomposition $x=(x_m,x_p)$.
Given an $x\in X$, it is at times useful to decompose it as $x = (x_p,x_m)$, where $x_m = (x_i^k)_{i=1,\ldots,\tilde{N},~ k=1,\ldots,\gamma_i-1}$ and $x_p$ contains the remaining components of $x$. (The subscript of $x_m$ is indicative of ``mixed strategy components'' of $x$ and $x_p$ indicative of ``pure strategy components'' of $x$.) In this decomposition, $x_m$ is a $\gamma$-dimensional vector containing the free components of $x$.
Taking the Jacobian of $F$ in terms of the components of $x_m$ we find that
\begin{equation} \label{eq_deriv_F_equals_hessian}
D_{x_m} F(x_p,x_m,u) = \tilde{\vH}(x),
\end{equation}
where $D_{x_m} F(x_p,x_m,u) = \left(\frac{\partial F}{\partial x_i^\ell}\right)_{\substack{i=1,\ldots,{\tilde{N}},\\ \ell=1,\ldots,\gamma_i-1}}$.

Let $x^*$ be a mixed equilibrium with carrier $C$. Differentiating \eqref{eq_potential_expanded_form2} we see that at the equilibrium $x^*$ we have $\frac{\partial U(x^*)}{\partial x_i^k} = 0$,
% for $i=1,\ldots,\tilde{N}$, $k=1,\ldots,\gamma_i-1$,
%(see Lemma \ref{lemma_carrier_vs_gradient} in appendix),
or equivalently,
\begin{equation} \label{eq_equilibrium_equation}
F_i^k(x^*,u) = U(y_i^{k+1},x^*_{-i}) - U(y_i^{1},x^*_{-i}) = 0
\end{equation}
for $i=1,\ldots,\tilde{N},~ k=1,\ldots,\gamma_i-1$.
Using \eqref{def_potential_fun} in the above we get
\begin{align}\label{eq_F_equality}
F_i^k(x^*,u)
%& = U(y_i^k,x^*_{-i}) - U(y_i^{K_i},x^*_{-i})%& = \sum_{m=1}^K q^m_i(e_i^k)\prod_{j\not = i} q^m(x_{j}) - \sum_{m=1}^K q^m_i(e_i^1)\prod_{j\not = i} q^m(x_{j})\\
 = \sum_{\tau=1}^K u^\tau \left[\left(q^\tau_i(\actionikone)- q^\tau_i(\actionione)\right)\prod_{j\not = i} q^\tau_j(x^*_{j})\right] =0.
\end{align}

Note that \eqref{eq_F_equality} is a linear in $u$. We would like to express \eqref{eq_F_equality} as a matrix equation (i.e., $\vA(x)u = 0$ for some matrix $\vA(x)$ dependent only on $x$).
With this in mind, it will be convenient to develop notation relating the ordering of the pure-strategy set $Y$ with the ordering of $(F_i^k)_{i=1,\ldots,\tilde{N},~k=1,\ldots,\gamma_i-1}$. Given $i\in\{1,\ldots,\tilde{N}\}$, $k\in\{1,\ldots,\gamma_i-1\}$, let
$$
s^*(i,k) :=
\begin{cases}
k & \mbox{ for }  ~ i=1,\\
\sum_{j=1}^{i-1} (\gamma_j-1) + k & \mbox{ for }  ~ i\geq 2.
\end{cases}
$$
%Given any $s$, there is a unique pair $(i,k)$, $i=1,\ldots,\tilde{N}$, $k=1,\ldots,\gamma_i-1$ such that $s=s^*(i,k)$. Thus, given $s=1,\ldots,\gamma$, we may also define
%\begin{align*}
%i^*(s) = i,\quad \mbox{ for the unique }~ i~ \mbox{ s.t. } s^*(i,k) = s \mbox{ for some }~ k,\\
%k^*(s) = k, \quad \mbox{ for the unique } ~k~ \mbox{ s.t. } s^*(i,k) = s \mbox{ for some }~ i.
%\end{align*}
Define $i^*:\{1,\ldots,\gamma\}\to \{1,\ldots,\tilde{N}\}$ and $k^*:\{1,\ldots,\gamma\}\to \{1,\ldots,\max_i\{\gamma_i-1\}\}$ to be the inverse of $s^*$; that is
\begin{equation} \label{def_k_i_star}
s^*(i^*(s),k^*(s)) = s
\end{equation}
for all $s=1,\ldots,\gamma$.

The above notation may be understood as follows: Effectively, we have stacked $(F_i^k)_{i=1,\ldots,\tilde{N},~k=1,\ldots,\gamma_i-1}$ into a single vector. Given a pair $(i,k)$, the function $s(i,k)\in\{1,\ldots,\gamma\}$ gives the corresponding index in this vector. Conversely, given some $s\in\{1,\ldots,\gamma\}$, the functions $i^*(s)$ and $k^*(s)$ simply return the player $i$ or action $k$ corresponding to entry $s$ in the vector.

Given an $x\in X$, let $\vA(x) = \left(a_{s,\tau}(x)\right)_{\substack{s = 1,\ldots,\gamma,\\ \tau=1,\ldots,K}}\in\mbb{R}^{\gamma\times K}$ be defined as the matrix with entries
\begin{equation}\label{def_A}
a_{s,\tau}(x) := \left(q^\tau_{i^*(s)}(y_i^{k^*(s)+1})- q^\tau_{i^*(s)}(\actionione)\right)\prod_{j\not = i^*(s)} q^\tau_j(x_{j}),
\end{equation}
%with $i = i^*(s)$ and $k=k^*(s)$.
Reexpressing \eqref{eq_F_equality} using this notation, we see that for any equilibrium $x^*$ with carrier $C$ we have
\begin{equation} \label{eq_A_equals_0}
\vA(x^*)u = 0.
\end{equation}

\section{General Games: Discussion and Proof Techniques} \label{sec:general-games}
In this section we take a short digression to discuss the proof that almost all general games are regular, and contrast this with the potential games case.
%regularity in general $N$-player games and contrast this with the potential games case.
%We will focus on contrasting the proof techniques in each case.
We note that this section may be skipped without loss of continuity.

In Section~\ref{sec:def-regularity-standard} we recall the traditional definition of a regular equilibrium. In Section~\ref{sec:gen-games-proof-strategy} we discuss the strategy used by Harsanyi \cite{harsanyi1973oddness} to prove that almost all general games are regular. In Section~\ref{sec:pot-game-proof-technique} we discuss the technical challenges arising in the case of potential games and contrast our proof techniques with those used by Harsanyi.

\subsection{Regular Equilibria in General Games} \label{sec:def-regularity-standard}
We now recall the traditional definition a regular equilibrium as given in \cite{van1981regular,van1991stability}. Let the game size $(N,(K_i)_{i=1}^N)$ be fixed.

As discussed in Section~\ref{sec_almost_all}, a game is uniquely defined by a vector $u\in \R^{NK}$ (which we refer to as the utility coefficient vector) specifying the pure-strategy utility received by each player.
%The vector $u$ uniquely defines the expected utility $U_i$ for each player $i$.

Given a strategy $x\in X$ and vector of utility coefficients $u\in\R^{NK}$, let
\begin{align} \label{F_tilde_def}
\tilde F_i^k(x,u) & := T_i^1(x_i)x_i^k [ U_i(y_i^{k+1},x_{-i}) - U_i(y_i^1,x_{-i})]
%\nonumber & = T_i^1(x_i)x_i^k F_i^k(x,u),
\end{align}
for $i=1,\ldots,N$, $k=1,\ldots,K_i-1$,
%where $F_i^k(x,u)$ is defined as in \eqref{def_F}.
and let
\begin{equation} \label{F_tilde_def2}
\tilde F(x,u) := \left( \tilde F_i^k(x,u) \right)_{\substack{i=1,\ldots,N\\ k=1,\ldots,K_i-1}}.
\end{equation}

\begin{definition} [Regular Equilibrium] \label{def_regular_equilib}
Let $x^*\in X$ be an equilibrium of a game with utility coefficient vector $u$.
% of a potential game with potential coefficient vector $u$.
Assume the action set  $Y_i$ of each player is reordered so that $y_i^1 \in \carr_i(x_i^*)$. The equilibrium $x^*$ is said to be \emph{regular} if the Jacobian of $\tilde F(x^*,u)$, given by $D_{x} \tilde F(x^*,u)$, is non-singular.
\end{definition}

\begin{remark}
We note that if $x^*$ is regular, then the Jacobian of $\tilde F(x^*,u)$ can be shown to be nonsingular under any reordering of $Y_i$  in which the reference action satisfies $y_i^1 \in \carr_i(x_i^*)$ for all $i=1,\ldots,N$ (see \cite{van1981regular}, Theorem 3.8). This justifies the use of an arbitrary reference action $y_i^1\in\carr_i(x_i^*)$ in the above definition.
\end{remark}

\begin{remark} \label{remark_reg_X_v_Delta}
The notion of a regular equilibrium is traditionally defined by considering mixed strategies directly in the probability simplex rather than $X_i$ \cite{van1991stability}. Using the definition of $\tilde F_i^k$ and the properties of the determinant of a matrix, it is readily confirmed that the definition of regularity given in Definition~\ref{def_regular_equilib} coincides with the traditional definition in \cite{van1991stability}.
\end{remark}

\subsection{Almost All General Games are Regular: Proof Strategy} \label{sec:gen-games-proof-strategy}
We will now review, at a high level, the classical technique for proving that almost all general games are regular given in \cite{harsanyi1973oddness}.

The strategy is as follows.
\begin{enumerate}
  \item Suppose that a carrier set $C$ is fixed.
  \item Let $u\in \R^{NK}$ denote a vector of utility function coefficients, and let $u$ be broken down as $u = (u^*,u^{**})$, where $u^*\in \R^\gamma$, $u^{**}\in \R^{NK-\gamma}$.
  \item Construct a $C^1$ function $\rho^{**}:\rOmega_C\times \R^{NK-\gamma}\to\R^{\gamma}$ such that, given any equilibrium $x^*$ with carrier $C$ and a partial vector of utility coefficients $u^{**}\in \R^{NK-\gamma}$, we have\footnote{The notation $\rho^{**}$, $\rho^*$, $u^*$ and $u^{**}$ is used here to be consistent with the usage in \cite{harsanyi1973oddness}.}
      $$
      u^* = \rho^{**}(x^*,u^{**}).
      $$
      That is, $\rho^{**}$ allows one to recover $u^*$ given $x^*$ and $u^{**}$.
      Given $\rho^{**}$, it is trivial to construct a $C^1$ function $\rho^{*}:\rOmega_C\times \R^{NK-\gamma}\to\R^{NK}$ such that
      $$
      u = \rho^*(x^*,u^{**}),
      $$
      so that $\rho^*$ recovers the full vector $u$ given $x^*$ and $u^{**}$.
  \item Show that the set of all irregular games having an equilibrium with carrier $C$ lies in the ``critical values'' set of $\rho^*$ (i.e, the image of the set of critical points). By Sard's theorem, this implies that all such games lie in an $\calL^{NK}$-measure zero set.
  \item Since the argument holds for arbitrary carrier $C$ and there are a finite number of possible carriers, this completes the proof.
\end{enumerate}

The crucial task in the above proof outline is the construction of the function $\rho^{**}$ which isolates the set of irregular games in a critical values set. For the sake of comparison with the identical payoff games case, we will now review the technique for constructing $\rho^{**}$ in \cite{harsanyi1973oddness} more detail.
%As this is the key part of the proof, we will review it in slightly more detail.  %so we may contrast with the techniques required to prove the analogous result for potential games.

Suppose a carrier $C$ is fixed, and suppose $x^*$  is an equilibrium with carrier $C$. Differentiating (using the same reasoning used to arrive at \eqref{eq_equilibrium_equation}) we see that at $x^*$ we have
\begin{equation} \label{eq_gen_eqilib_condition}
U_i(y_i^{k+1},x_{-i}^*) - U_i(y_i^1,x_{-i}^*) = 0
\end{equation}
for $i=1,\ldots,\tilde N$, $k=1,\ldots,\gamma_i-1$. Let $u=(u_1,\ldots,u_N)\in\R^{NK}$ represent the vector of all player's pure-strategy utilities.
As with \eqref{eq_equilibrium_equation}, the equality \eqref{eq_gen_eqilib_condition} is linear in $u$ and may be expressed as a matrix equation\footnote{To be consistent with the proof in the potential games case, we have modified the presentation of Harsanyi's technique to accommodate matrix notation. However, modulo notational differences, the argument discussed here is the same as \cite{harsanyi1973oddness}.}
\begin{equation} \label{eq_B_equals_zero}
\vB(x)u = 0,
\end{equation}
where $\vB(x)\in \R^{\gamma\times NK}$. The matrix $\vB(x)$ has particularly convenient structure. To elucidate this structure, first note that, as with \eqref{def_potential_fun}, the (expected) utility of player $i$ in a general game may be expressed as
$$
U_i(x) = \sum_{\tau=1}^K u_i^\tau \left[\prod_{i=1}^N q_i^\tau(x_i)\right].
$$
Using this form, the equality \eqref{eq_gen_eqilib_condition} is expressed as
\begin{equation} \label{eq_row-sum-general}
\sum_{\tau=1}^K u_i^\tau \left[\left(q^\tau_i(\actionikone)- q^\tau_i(\actionione)\right)\prod_{j\not = i} q^\tau_j(x^*_{j})\right] =0,
\end{equation}
for $i=1,\ldots,\tilde N$, $k=1,\ldots,\gamma_i-1$.
%Let $s^*(i,k)$ be as defined in ?? and
For each $i\in \{1,\ldots,\tilde N\}$, let $\vB_i(x) = (b_{i,k,\tau}(x))_{\substack{k=1,\ldots,\gamma_i-1\\\tau=1,\ldots,K}}$ be the matrix with entries
\begin{equation} \label{def_B}
b_{i,k,\tau}(x) := \left(q^\tau_{i}(y_i^{k+1})- q^\tau_{i}(\actionione)\right)\prod_{j\not = i} q^\tau_j(x_{j}).
\end{equation}
The equality \eqref{eq_row-sum-general} may now be expressed in the form of \eqref{eq_B_equals_zero} with $\vB(x)$ given by
$$\vB(x) :=
\begin{pmatrix}
\vB_1(x) & \vzero & \cdots & \vzero\\
\vzero & \vB_2(x)  & \cdots & \vzero\\
\vdots & \ddots & ~ & ~ \vdots \\
\vzero & \cdots & \vzero & \vB_{\tilde N}(x)
\end{pmatrix}.
$$
The important fact about $\vB(x)$ is that there exist $\gamma$ columns of $\vB(x)$ so that these same columns form an invertible diagonal $\gamma\times \gamma$ submatrix of $\vB(x)$, for all $x$ with carrier $C$. (More details can be found in \ref{Appendix-general-extras}.)
Thus, without loss of generality, we may reorder the set of pure strategies so that the first $\gamma$ columns of $\vB(x)$ are linearly independent and we may write $\vB(x) = [\vC(x)~ \vD(x)]$, where $\vC(x)\in \R^{\gamma\times \gamma}$ is invertible for any $x$ with carrier $C$.

Given a vector $u^{**}\in \R^{NK-\gamma}$ and an equilibrium $x^*$ with carrier $C$, define
\begin{equation} \label{eq_rho-general-construction}
\rho^{**}(x,u^{**}) := \vC(x)^{-1}\vD(x)u^{**}.
\end{equation}
Finally, recalling \eqref{eq_B_equals_zero}, we observe that $\rho^{**}$ allows one to recover $u^*$ given $u^{**}$ and any equilibrium $x^*$ with support $C$.

In the following section we will discuss the strategy for constructing an analogous function in the case of identical payoff games, and the challenges that arise in this case.

\subsection{Comparison with Proof Technique in Identical-Payoff Games} \label{sec:pot-game-proof-technique}
We will now contrast the proof that almost all general games are regular with the proof that almost all identical payoff games are regular.

The main substantive difference between the proof of Proposition~\ref{lemma_second_order_degeneracy} and proof of generic regularity in the case of general games is the construction the functions isolating the subset of irregular games (Step 3 in Section~\ref{sec_roadmap} and Step 3 in Section~\ref{sec:gen-games-proof-strategy}).
We will now briefly discuss the manner in which such a function is constructed in the identical-payoffs-game case and contrast this with the approach discussed in Section~\ref{sec:gen-games-proof-strategy}.

In the case of an identical-payoff game, at an equilibrium $x^*$ with carrier $C$ we have the key relationship \eqref{eq_A_equals_0} rather than \eqref{eq_B_equals_zero}. The matrices $\vA(x)$ and $\vB(x)$ are related via
\begin{equation} \label{eq_A_B_relation}
\vA(x) =
\begin{pmatrix}
\vB_1(x)\\
\vdots\\
\vB_N(x)
\end{pmatrix}.
\end{equation}
Suppose we fix some carrier $C$, and let $\gamma$ be as defined in \eqref{eq_gamma_def}.
Given a vector of potential function coefficients $u\in\R^{K}$, decompose it as $u = (u_1,u_2)$, where $u_1\in \R^{K-\gamma}$, $u_2\in\R^{\gamma}$.\footnote{In Section~\ref{sec:gen-games-proof-strategy} we used $u^*$ and $u^{**}$ to refer to an analogous breakdown of $u\in\R^{NK}$; this notation was preferred there to avoid confusion with individual player's utility functions.}

Analogous to the general games case, if we can find a function that recovers $u_2$ given $u_1$ and an equilibrium $x^*$ with carrier $C$, then the set of second-order degenerate potential games will lie in the critical values set of this map (see Lemma~\ref{lemma_critical_vals}). %(Due to the relationship \eqref{eq_A_B_relation} and the linear nature of both \eqref{eq_B_equals_zero} and \eqref{eq_A_equals_0}, this will follow from the same reasoning used to show that the set of irregular games lie in the critical values set of $\rho^{**}$ in \cite{harsanyi1973oddness}. See Lemma \ref{lemma_critical_vals} below for more details.)
%in Section \ref{sec:gen-games-proof-strategy}.)
Thus, the key step in the proof of Proposition~\ref{lemma_second_order_degeneracy} is the construction of such a function. %$\rho$.

In the case of general $N$-player games we were conveniently able to choose $\gamma$ columns of $\vB(x)$ that formed a $\gamma\times \gamma$ diagonal submatrix of $\vB(x)$ that was invertible for all $x$ with carrier $C$. Using this diagonal submatrix we were able to construct an appropriate function $\rho^{**}$ in \eqref{eq_rho-general-construction}.

In the case of identical payoff games, we can (and will) follow a similar approach. If $\vA(x)$ can be decomposed as $\vA(x) = [\vE(x) ~ \vF(x)]$, where $\vE(x) \in\R^{\gamma\times \gamma}$ is invertible, then (as in \eqref{eq_rho-general-construction}) the function $\rho(x,u_1) = \vE(x)^{-1} \vF(x)u_1$ will recover $u_2$.
However, in the case of identical payoff games, the problem of finding an invertible submatrix of $\vA(x)$ is more challenging than the general games case. Unlike $\vB(x)$, there do not exist $\gamma$ columns of $\vA(x)$ that form a invertible diagonal submatrix of $\vA(x)$. Moreover, the pure strategies (corresponding to columns of $\vB(x)$) used to construct Harsanyi's map $\rho^{**}$ need not correspond to linearly independent columns of $\vA(x)$.

%In the case of identical payoff games, no such diagonal matrix can be formed. In particular, the matrix $\vA(x)$ does not have $\gamma$ columns which form a diagonal submatrix. Thus, the selection of $\gamma$ linearly independent columns is more challenging. In particular, the pure strategies (corresponding to columns of $\vB(x)$) used to construct Harsanyi's map $\rho^{**}$ may not correspond to linearly independent columns of $\vA(x)$.

%the primary difficulty in recovering $u_2$ from $x$ and $u_1$ using the technique discussed in Section \ref{sec:gen-games-proof-strategy} is that the matrix $\vA(x)$ does not possess $\gamma$ columns such that these same columns are linearly independent for all $x$ with carrier $C$.
For example, suppose $\Gamma$ is a $2\times 2$ game and let $C$ be the carrier containing all pure strategies (so any $x$ with $\carr(x)=C$ is completely mixed). The matrix $\vA(x)$ takes the form
\begin{equation} \label{eq_A_example}
\vA(x) =
\begin{pmatrix}
-x_2 & -x_2 & (1-x_2) & (1-x_2)\\
-x_1  & (1-x_1) & -x_1 & (1-x_1)
\end{pmatrix}.
\end{equation}
(See \ref{Appendix-A-full-rank} and, in particular, \eqref{eq_alpha_q_relationship}--\eqref{def_p_partition} for an in-depth characterization of the structure of $\vA(x)$.)
%The matrix $\vB(x)$ takes the form
%\begin{equation} \label{eq_B_example}
%\vB(x) =
%\begin{pmatrix}
%-x_2 & -x_2 & (1-x_2) & (1-x_2) & 0 & 0 & 0 & 0\\
%0 & 0 & 0 & 0 & -x_1  & (1-x_1) & -x_1 & (1-x_1)
%\end{pmatrix}.
%\end{equation}
%The mapping used by Harsanyi in \cite{harsanyi1973oddness} selects the columns of $\vB(x)$ corresponding to the pure strategies $(y_1^1,y_2^2)$ and $(y_1^2,y_2^1)$. In \eqref{eq_B_example} this corresponds to the second and forth columns.
In \eqref{eq_A_example} these pure-strategy choices correspond to the middle two columns of $\vA(x)$, which are linearly dependent for $x_1=x_2=1/2$.
In order to ensure that we may select $\gamma$ columns of $\vA(x)$ that are linearly independent, we will show that for any $x$ with $\carr(x) \subseteq C$, the matrix $\vA(x)$ has full rank (see Proposition~\ref{proposition_A_full_rank}). This will then allow us to construct a set of functions as in Step 3 of Section~\ref{sec_roadmap} (see also \eqref{eq_rho_def1}) that recover $u_2$ given $u_1$ and $x$ in some subset of the strategy space.

It should be noted that, for the sake of proving Proposition~\ref{lemma_second_order_degeneracy}, it would be sufficient to prove the forthcoming Proposition~\ref{proposition_A_full_rank} under the weaker condition $\carr(x)= C$ (rather than $\carr(x)\subseteq C$). However, we will prove Proposition~\ref{proposition_A_full_rank} under the stronger condition that $\carr(x) \subseteq C$ in order to apply the result later in the proof of Proposition~\ref{lemma_first_order_degeneracy}.

\section{Proof of Proposition~\ref{lemma_second_order_degeneracy}} \label{sec:proof-2nd-order-degen}
We will now prove Proposition~\ref{lemma_second_order_degeneracy}, which states that almost all identical payoff games are second-order non-degenerate.

We begin with the following proposition that establishes the solvability of \eqref{eq_A_equals_0}.
%and will permit us to construct a family of functions $(\rho_\ell)_\ell$ as outlined in Section \ref{sec_roadmap}. We will proceed with this construction after pro
%We note that, in order to prove Proposition \ref{lemma_second_order_degeneracy} it is sufficient to consider only $x$ such that $\carr(x) = C$ in the following proposition. However, later, when studying first order degenerate games in Section \ref{sec_first_order_degenerate_games}, we will consider the notion of an ``extended carrier set'' (see \eqref{def_ext_carr}), and we will need to characterize the rank of $\vA(x)$ for $x$ with $\carr(x) \subset C$.
\begin{proposition} \label{proposition_A_full_rank}
For any $x$ such that $\carr(x) \subseteq C$, the matrix $\vA(x)$ has full row rank.
%For any $x$ with $\supp(x)\subseteq C$ the matrix $\vA(x)$ has full row rank.
\end{proposition}
Proposition~\ref{proposition_A_full_rank} together with \eqref{eq_A_equals_0} will permit us to construct a family of functions $(\rho_\ell)_{\ell\geq 1}$ as outlined in Section~\ref{sec_roadmap}.
%We will proceed with the construction of these functions after proving Proposition \ref{proposition_A_full_rank}.
The complete proof of Proposition~\ref{proposition_A_full_rank} can be found in \ref{Appendix-A-full-rank}.

In broad strokes, the proof of Proposition~\ref{proposition_A_full_rank} proceeds as follows: We will consider the ``sign pattern'' of the matrix $\vA(x)$ (i.e., the matrix $\sgn(\vA(x))$, with $\sgn(\cdot)$ defined in \ref{Appendix-math-notation}).
%If $x$ and $x'$ are two strategies with $\carr(x) = \carr(x')$, then $\sgn(\vA(x)) = \sgn(\vA(x'))$. Thus, for $x\in $\sgn(\vA(x))$
A matrix is called a \emph{sign-pattern matrix} if its entries take on values from the set $\{-1,0,1\}$. The properties of sign-pattern matrices have been well studied in the literature \cite{MR1273974,MR743051}. Of relevance here, a sign pattern matrix $\vM$ is called an \emph{$L$-matrix} if every matrix with the same sign pattern as $\vM$ has full rank. A relatively simple characterization of $L$-matrices is given in \cite{MR1273974,MR743051}. Using this characterization, we prove that for any $x$ with $\carr(x) \subseteq C$, $\sgn(\vA(x))$ is an $L$-matrix. This implies the desired result.

Please see \ref{Appendix-A-full-rank} for complete details of the proof of Proposition~\ref{proposition_A_full_rank}.

We will now use Proposition~\ref{proposition_A_full_rank} to construct a countable cover $(B_\ell)_{\ell\geq 1}$ of the set $\{x:\carr(x) =C\}$ and a family of functions $(\rho_\ell)_{\ell\geq 1}$ as outlined in Section~\ref{sec_roadmap}.

Given the carrier $C$ there are $\binom{K}{\gamma}$ possible combinations (of size $\gamma$) of the columns of $\vA(x)$. For each $r=1,\ldots,\binom{K}{\gamma}$, let $\vA_r(x)\in\mbb{R}^{\gamma\times\gamma}$ denote a square matrix formed by taking one unique combination of the columns of $\vA(x)$. We have the following lemma.
\begin{lemma} \label{lemma_cover}
There exists a collection of open balls $(B_\ell)_{\ell\geq 1}$, $B_\ell\subset \{x:\carr(x) = C\}$ that satisfy:\\
(i) $\bigcup_{\ell\geq 1} B_\ell = \{x\in X:~\carr(x) = C\}$\\
(ii) For each $\ell\in\N$ there exists an $r_\ell\in\{1,\ldots,\binom{K}{\gamma}\}$ such that $\vA_{r_\ell}(x)$ is invertible for all $x\in B_\ell$.
\end{lemma}
\begin{proof}
%Let $p_r(x)$ denote the characteristic polynomial of $\vA_r(x)$.
For $r=1,\ldots,\binom{K}{\gamma}$, let
$$
S_r:= \{x\in X:~ \carr(x) = C,~\det \vA_r(x) \not= 0\}.
$$
%Note that whenever $x\not\in S_r$ the matrix $\vA_r(x)$ is non-singular.
By Proposition~\ref{proposition_A_full_rank}, no strategy $x\in X$ with $\carr(x) = C$ may simultaneously be in all $S_r^c$. Hence $\bigcup_r S_r= \{x:\carr(x) = C\}.$
%Let $\Omega_C$ be as defined in \eqref{def_C_face}.
%$$
%\Omega_C:=\{x\in X:~ T_i^k(x_i)\geq 0,~i=1,\ldots,\tilde{N},k=1,\ldots,\gamma_i\}.
%$$
Note also that each $S_r$ is open relative to $\{x\in X:~\carr(x) = C\}$. Thus, for each $r$ we may construct a countable cover $(B_{r,\ell})_{\ell\geq 1}$ of $S_r$, such that $B_{r,\ell} \subset \{x:\carr(x) = C\}$, $\ell\in\N$ and $\bigcup_\ell B_{r,\ell} = S_r$.

By construction we have (i) for any pair $(r,\ell)$, the matrix $\vA_r(x)$ is invertible for all $x\in B_{r,\ell}$, and (ii) $\bigcup_{r,\ell} B_{r,\ell} = \bigcup_r S_r = \{x:\carr(x) = C\}$. Hence, $(B_{r,\ell})_{r=1,\ldots,\binom{K}{\gamma},\,\ell\in\N}$ is a countable cover with the desired properties.
\end{proof}

%\begin{proof}
%For each $r\in 1,\ldots,\binom{K}{\gamma}$, $S_r^c$ is open and there exists a countable cover $(B_{r,\ell})_{\ell}$ of $S_i^c$. Taking the union $\bigcup_{r,\ell} B_{r,\ell}$ and relabeling with a single index, we get a countable cover $(B_\ell)_\ell$ with the desired properties.
%\end{proof}

Fix $\ell \in \{1,2,\ldots\}$.
%Fix $r=1,\ldots,\binom{K}{\gamma}$ and
After reordering, $\vA(x)$ may be partitioned as
$\vA(x) = $ \newline $ [\tilde \vA_{r_\ell}(x) ~ \vA_{r_\ell}(x)]$,
where $\tilde \vA_{r_\ell}(x)$ is a matrix formed by the columns of $\vA(x)$ not used to form $\vA_{r_\ell}(x)$. Let the strategy set $Y$ be reordered in the same way as the columns $\vA(x)$.\footnote{Note that we previously assumed a specific ordering for $Y$. However, this was for the purpose of proving Proposition~\ref{proposition_A_full_rank}, which is unaffected by a reordering of $Y$ at this point.}
Given a vector of utility coefficients $u\in\R^K$, let it be partitioned as $u = (u_1,u_2)$, where $u_1 = (u^1,\ldots,u^{K-\gamma})$ and $u_2 = (u^{K-\gamma+1},\ldots,u^{K}$). Define $\tilde{\rho}_\ell:B_\ell\times \R^{K-\gamma}\rightarrow \R^{\gamma}$ by
\begin{equation}\label{eq_rho_def1}
\tilde{\rho}_\ell(x,u_1) := -\vA_{r_\ell}(x)^{-1} \tilde \vA_{r_\ell}(x) u_1,
\end{equation}
If $x^*\in B_\ell$ is an equilibrium for some identical-payoffs game with utility coefficient vector $u$, then
%Given an equilibrium $x^* \in B_\ell$,
by \eqref{eq_A_equals_0} we have $\vA(x^*)u = 0$. Since $\vA_{r_\ell}(x^*)$ is invertible, this is equivalent to
$u_2 = -\vA_{r_\ell}(x^*)^{-1} \tilde \vA_{r_\ell}(x^*) u_1$.
%For any $u_1$, the unique solution to this equation is given by $u=(u_1,u_2)$, $u_2 = \tilde{\rho}_\ell(x^*,u_1)$.
Hence, if $x^*\in B_\ell$ is an equilibrium of some identical-payoffs game with utility coefficient vector $u=(u_1,u_2)$, the function $\tilde{\rho}_\ell$ permits us to recover $u_2$ given $u_1$ and $x^*$.

For each $\ell\in\N$, define the function $\rho_\ell:B_\ell\times\R^{K-\gamma}\rightarrow \R^{K}$ by
\begin{equation} \label{eq_rho_def2}
\rho_\ell(x,u_1) := (u_1,\tilde{\rho}_\ell(x,u_1)).
\end{equation}
The function $\rho_\ell$ is a trivial extension of $\tilde{\rho}_\ell$ that recovers the full vector of utility coefficients $u\in\R^{K}$ given $u_1\in \R^{K-\gamma}$ and an equilibrium $x^*\in B_\ell$.

We thus get the following corollary to Lemma~\ref{lemma_cover}.
\begin{corollary} \label{cor_cover}
There exists a countable collection of open balls $(B_\ell)_{\ell\geq 1}$, $B_\ell\subset \rOmega_C$ such that\\
(i) $\bigcup_{\ell\geq 1} B_\ell = \{x\in X:~\carr(x) = C\}$\\
(ii) For each $\ell\in\N$, there exists a differentiable function $\rho_\ell:B_\ell\times \R^{K-\gamma}\to\R^K$ such that, if $x^*\in B_\ell$ and $x^*$ is an equilibrium of the game with utility coefficients $u = (u_1,u_2)$, $u_1\in\R^{K-\gamma}$, $u_2\in\R^\gamma$, then $u = \rho_\ell(x^*,u_1)$.
\end{corollary}

We recall that our end goal is to use Sard's theorem to show that the set of second-order degenerate games has $\calL^K$-measure zero. To that end, we will now show that for each $\ell\in \N$, the set of identical payoff games that have a second-order degenerate equilibrium lying in $B_\ell$ is contained in the critical values set of $\rho_\ell$.

Suppose $x\in B_\ell$ and $u_1 \in \R^{K-\gamma}$ are arbitrary. If $u=(u_1,u_2)$ with $u_2 = \tilde{\rho}_\ell(x,u_1)$, then by the definition of $\tilde{\rho}_\ell$ we see that $\vA(x)u=0$.
By the definition of $\vA(x)$ (see \eqref{def2_F}--\eqref{def_A}) this implies
$$
F(x,u_1,\tilde{\rho}_\ell(x,u_1)) = 0, ~\quad \mbox{ for all } x\in B_\ell.
$$
%for all $x\in B_\ell$.
Thus, taking a partial derivative with respect to $x_i^k$ we get
\begin{equation}\label{eq1}
\frac{\partial}{\partial x_i^k} F(x,u_1,\tilde{\rho}_\ell(x,u_1)) = 0, \quad\quad i=1,\ldots,\tilde{N}, ~k=1,\ldots,\gamma_i-1.
\end{equation}
Consider again the decomposition $x=(x_p,x_m)$.
Using compact notation, \eqref{eq1} is restated as
\begin{equation} \label{eq_partial_jacobian_F}
D_{x_m} F(x_p,x_m,u_1,\tilde{\rho}_\ell(x_p,x_m,u_1)) = 0.
\end{equation}
Suppose $x^*\in B_\ell$ is an equilibrium of an identical-payoffs game with utility coefficient vector $u$ and $\carr(x^*)=C$.
Applying the chain rule in \eqref{eq_partial_jacobian_F}, using \eqref{eq_deriv_F_equals_hessian}, and using the fact that $F(x,u) = \vA(x)u = \tilde{\vA}_{r_\ell}(x)u_1 + \vA_{r_\ell}(x)u_2$, we find that at $x^*$ there holds\footnote{We note that $\tilde{\vH}(x^*)$ is also dependent on $u$. However, we suppress the argument $u$ since it is generally held constant in the context of the Hessian $\tilde{\vH}$.}
\begin{equation} \label{non-degeneracy_eq1}
\tilde{\vH}(x^*) = -\vA_{r_\ell}(x^*) \tilde{\vJ}_{\rho_\ell}(x^*,u),
\end{equation}
where $\tilde{\vJ}_{\rho_\ell}(x,u) := D_{\tilde{x}_m}\tilde{\rho}_\ell(x_p,\tilde{x}_m,u)\big\vert_{\tilde{x}_m = x_m}$ is the Jacobian of $\tilde{\rho}_\ell$ taken with respect to $x_m$.

Since $\vA_{r_\ell}(x)$ is invertible for all $x\in B_\ell$, this means that given any equilibrium $x^*\in B_\ell$,
%given any $x^*$ with carrier $C$ if we choose $r$ such that $\vA_r(x^*)$ is invertible (which we may always do), then
the Hessian $\tilde{\vH}(x^*)$ is nonsingular if and only if the Jacobian $\tilde{\vJ}_{\rho_\ell}(x^*,u)$ is nonsingular.

Note that the Jacobian of $\rho_\ell$ takes the form
$$
\vJ_{\rho_\ell}(x,u) =
\begin{pmatrix}
{\bf{0}} & \vI \\
\vJ_{\tilde{\rho}_\ell} & \vM
\end{pmatrix},
$$
for some matrix $\vM$. Clearly, $\det \vJ_{\rho_\ell} = 0$ if and only if $\det \tilde{\vJ}_{\rho_\ell} = 0$. Thus, by \eqref{non-degeneracy_eq1} we get the following result.
\begin{lemma} \label{lemma_critical_vals}
If $x^*\in B_\ell$ is an equilibrium of an identical-payoffs game with utility coefficient vector $u$, then
%at an equilibrium $x^*\in B_\ell$  we have
\begin{equation}
\det \vJ_{\rho_\ell}(x^*,u) = 0 \quad  \iff \quad \det \tilde{\vH}(x^*) = 0,
\end{equation}
where $\vJ_{\rho_\ell}$ is the Jacobian of $\rho_\ell$.
\end{lemma}

In particular, the above result shows that if $u$ is an identical payoffs game having a second-order degenerate equilibrium $x^*\in B_\ell$, then $u$ is contained in the set of critical values of $\rho_\ell$.
We now prove Proposition~\ref{lemma_second_order_degeneracy}.
\begin{proof}
Let $C$ be a carrier set. Let $\calU(C) \subseteq \R^K$ be the set of identical-payoff games having at least one degenerate equilibrium with carrier set $C$; that is,
$$
\calU(C) := \big\{u\in \R^{K}:~ \exists \mbox{ degenerate equilibrium } x^*\in \{x\in X: \carr(x) = C\} \big\}.
$$

Let $(B_\ell)_{\ell\in \N}$ and $(\rho_\ell)_{\ell\in \N}$ be as in Corollary~\ref{cor_cover}. We note that $B_\ell$ and $\rho_\ell$, $\ell\in \N$ are implicitly defined with respect to $C$.
For $\ell \in \N$, let $\calU(C,\ell) \subseteq \R^K$ be the subset of identical-payoff games having at least one degenerate equilibrium $x^*\in B_\ell$; that is,
$$
\calU(C,\ell) := \big\{u\in \R^{K}:~ \exists \mbox{ degenerate equilibrium } x^*\in B_\ell \big\}.
$$

By construction, we have $\bigcup_{\ell\geq 1}B_\ell = \{x\in X: \carr(x) = C\}$, and hence $\calU(C) = \bigcup_{\ell\geq 1} \calU(C,\ell)$.

%For each $\ell$, there exists an $r$ such that the function $\rho_\ell:B_\ell \times \R^{K-\gamma}\rightarrow \R^{K}$ is well defined.

%We showed above that for any $(x,u)\in B_\ell\times \R^K$ such that $x$ is an equilibrium of the identical-payoffs game with utility coefficients $u$, the Hessian $\tilde{\vH}(x)$ (taken with respect to $C$) is invertible if and only if the Jacobian of $\rho_\ell(x,u)$ is invertible.
%We showed above that for any equilibrium $x^* \in B_\ell$ the Jacobian of $\rho_\ell(x^*,u):B_\ell\times\R^{K-N-\gamma}\rightarrow\R^{K-N}$ is singular if and only if the Hessian $\tilde{\vH}(x^*)$ is singular.
By Lemma~\ref{lemma_critical_vals}, the set $\calU(C,\ell)$ is contained in the set of critical values of $\rho_\ell$. By Sard's theorem, we get that $\calU(C,\ell)$ is a set with $\calL^K$-measure zero. Since $\calU(C)$ is the countable union of sets of $\calL^K$-measure zero, it is itself a set with $\calL^K$-measure zero.

Let $\calU\subset\R^K$ denote the subset of identical-payoff games with at least one degenerate equilibrium. The set $\calU$ may be expressed as the union $~\calU = \bigcup_C\calU(C)$ taken over all possible carrier sets $C$. Since there are a finite number of carrier sets $C$, the set $\calU$ has $\calL^K$-measure zero.
\end{proof}

\section{First-Order Degenerate Games} \label{sec_first_order_degenerate_games}
The following proposition shows that, within the set of identical-payoff games, first-order degenerate games form a null set.
%Combined with Lemma \ref{lemma_second_order_degeneracy}, this proves Theorem \ref{thrm_regular_eqilibria}.
\begin{proposition} \label{lemma_first_order_degeneracy}
%Almost all potential games are first-order non-degenerate.
The set of identical-payoff games which are first-order degenerate has $\calL^K$-measure zero.
\end{proposition}
\begin{proof}

Fix some set $C=C_1\cup\cdots\cup C_N$ where each $C_i$ is a nonempty subset of $Y_i$. Let $\widehat{C}$ be any \emph{strict} subset of $C$. In the context of this proof let $\gamma_i := |C_i|$, let $\gamma := \sum_{i=1}^{N} (\gamma_i-1)$, let $\tilde N := |\{i=1,\ldots,N: \gamma_i\geq 2 \}|$, and assume $Y_i$ is reordered so that $\{y_i^1,\ldots,y_i^{\gamma_i}\} = C_i$. Note that this ordering implies that for any $x$ with $\carr(x) = C$ we have $y_i^1\in \carr(x)$, $i=1,\ldots,N$.

Given an equilibrium $x^*$ let the \emph{extended carrier} of $x^*$ be defined as
\begin{equation}\label{def_ext_carr}
\extcarr(x^*):= \carr(x^*)\cup\left(\bigcup_{i=1}^N \{y_i^k\in Y_i: ~k=2,\ldots K_i, ~ \frac{\partial U(x^*)}{\partial x_i^{k-1}} = 0 \} \right).
\end{equation}
Suppose that $x^*$ is an equilibrium with \emph{extended} carrier $C$. By Lemma ~\ref{lemma_carrier_vs_gradient} (see appendix) and the ordering we assumed for $Y_i$ we have
\begin{equation} \label{IR6_eq1}
F_i^k(x^*,u) = \frac{\partial U(x^*)}{\partial x_i^{k-1}}= 0, \qquad i=1,\ldots,\tilde{N}, ~k=1,\ldots,\gamma_i-1,
\end{equation}
where $F_i^k$ is as defined in \eqref{def_F}. Thus, if $x^*$ is an equilibrium for some identical-payoffs game with utility coefficient vector $u\in\R^K$, and $\extcarr(x^*) = C$, then by the definition of $\vA(x)$ (see \eqref{def2_F}--\eqref{def_A}), \eqref{IR6_eq1} implies that $\vA(x^*)u = 0$, or equivalently,
$$
u\in \ker \vA(x^*),
$$
where the matrix $\vA(x)\in \R^{\gamma\times K}$ is defined with respect to $C$, as in \eqref{def_A}.

Let $\calU(C,\widehat{C})\subseteq \R^K$ be the set of identical-payoff games in which there exists an equilibrium $x^*$ with $\carr(x^*) = \widehat{C}$ and $\extcarr(x^*) = C$. Let
$$
\widehat{X} := \{x\in X:~\carr(x) = \widehat{C}\}.
$$
By the above we see that
\begin{equation} \label{IR6_eq2}
\calU(C,\widehat{C}) \subseteq \cup_{x\in \widehat{X}} \ker \vA(x).
\end{equation}
For each $x\in\widehat{X}$, let $\range \vA(x)^T$ denote the range space of $\vA(x)^T$.
Each entry of $\vA(x)^T$ is a polynomial function in $x$ and hence is Lipschitz continuous over the bounded set $\widehat{X}$.
By Proposition~\ref{proposition_A_full_rank} we have $\rank \vA(x)^T = \gamma$ for all $x\in \widehat{X}$.
%Since $\vA(x)$ is entrywise Lipschitz continuous in $x$ and since $\rank \vA(x) = \gamma$ for all $x\in \widehat{X}$,
Thus, we may choose a set of $\gamma$ basis vectors
$\{\vb_1(x),\ldots,\vb_\gamma(x)\}$ spanning $\range \vA(x)^T$ such that each $\vb_k(x)\in \R^{K}$, $k=1,\ldots,\gamma$ is a Lipschitz continuous function in $x$.
Moreover, we may choose a complementary set of $(K-\gamma)$ linearly independent vectors $\{\vb_{\gamma+1}(x),\ldots,\vb_{K}(x)\}$ forming a basis for the orthogonal complement $(\range \vA(x)^T)^{\perp}$, with each $\vb_k(x)\in\R^K$, $k=\gamma+1,\ldots,K$ being a continuous function in $x$. Let $\vB(x) := \left( \vb_{\gamma+1}(x)~ \cdots ~ \vb_K(x)\right)\in \R^{K\times (K-\gamma)}$.
Let $f:\widehat{X}\times \R^{K-\gamma} \rightarrow \R^K$ be given by $
f(x,v) := \vB(x)v.$ Since $\vB(x)$ is Lipschitz continuous in $x$ and $\widehat{X}$ is bounded, $f$ is Lipschitz continuous.
By the fundamental theorem of linear algebra, for each $x\in\widehat{X}$, $\ker \vA(x) = (\range \vA(x)^T)^{\perp} = \range \vB(x)$. Hence,
\begin{equation} \label{IR6_eq3}
\bigcup_{\substack{x\in \widehat{X}}} \ker \vA(x)
=  f(\widehat{X} \times \R^{K-\gamma}).
\end{equation}
Since $\widehat{C}\subsetneq C$, the Hausdorff dimension of $ \widehat{X}$ is at most $(\gamma-1)$ and the Hausdorff dimension of $\widehat{X}\times\R^{K-\gamma}$ is at most $K-1$. Since $f$ is Lipschitz continuous, this implies (see \cite{EvansGariepy}, Section 2.4) that the Hausdorff dimension of $f(\widehat{X} \times \R^{K-\gamma})$ is at most $K-1$, and in particular, that $f(\widehat{X} \times \R^{K-\gamma})$ has $\calL^K$-measure zero. By \eqref{IR6_eq2} and \eqref{IR6_eq3}, this implies that $\calU(C,\widehat{C})$ has $\calL^K$-measure zero.

Let $\calU \subseteq \R^K$ denote the set of all identical-payoff games containing a first-order degenerate equilibrium. Since we may represent this set as a finite union of $\calL^K$-measure zero sets,
$$
\calU=\bigcup_{\substack{\emptyset \neq C_i \subseteq Y_i,~i=1,\ldots,N\\ C=C_1\cup\cdots\cup C_N\\ \widehat{C}\subsetneq C  }} \calU(C,\widehat{C}),
$$
the set $\cal U$ also has $\calL^K$-measure zero.
\end{proof}

\section{Regularity in Exact and Weighted Potential Games} \label{sec_exact_and_weighted_games}

The goal of this section will be to prove the following proposition.
%goal of this section is to prove generic regularity in exact and weighted potential games, as stated in the following proposition.
\begin{proposition} \label{prop_exact_and_weighted}
$~$\\
\noindent (i) Almost all weighted potential games are regular.\\
\noindent (ii) Almost all exact potential games are regular.
\end{proposition}

%given a potential game with potential function $U$, we define the ``associated identical-payoffs game'' to be the identical-payoffs game in which the utility of each player is given by $U_i=U$. Clearly, an equilibrium in a potential game is regular if and only if it is a regular equilibrium in the associated identical interests game.

%The remainder of this section will be devoted to proving this proposition.
Exact and weighted potential games are closely related to games with identical payoffs. By identifying equivalence relationships between identical-payoff games and exact and weighted potential games, this result follows as a relatively simple consequence of Proposition~\ref{prop_identical_payoff}.

Let the number of players $N\geq 2$ and action-space size $K_i\geq 2$, $i=1,\ldots,N$ be arbitrary.
Define
$$
\tilde \calI := \{u\in \calI:~ \ones^T u = 0\},
$$
where $\calI = \R^K$ is the set of identical-payoff games as given Section~\ref{sec_almost_all}.
%in \eqref{def_calI}

The set $\tilde \calI$ will provide a convenient means of partitioning the sets of exact and weighted potential games.

Note that an identical-payoffs game $u\in \tilde \calI$ is regular if and only if $\tilde u=u+c\ones$ is regular for every $c\in\R$.
Note also that $\tilde \calI$ is isomorphic to $\R^{K-1}$. (Henceforth we will treat $\tilde \calI$ as $\tilde \calI = \R^{K-1}$.) This, along with Proposition~\ref{prop_identical_payoff}, implies the following lemma.
\begin{lemma} \label{lemma_reduced_u_regular}
The set of games in $\tilde \calI$ that are irregular has $\calL^{K-1}$-measure zero.
%Almost all $u\in \tilde \calI$ are regular.
\end{lemma}

%As defined in Section \ref{sec_almost_all}, let $\calP\subset \R^{NK}$ denote the set of all exact potential games and let $\calW\subset \R^{NK}$ denote the set of all weighted potential games (of size $(N,(K_i)_{i=1}^N)$).

The sets of weighted and exact potential games (of size $(N,(K_i)_{i=1}^N)$) are defined explicitly as follows.
For convenience, we decompose the vector of utility coefficients $u\in\R^{NK}$ as $u = (u_i)_{i=1}^N$, where $u_i\in\R^K$ gives the pure strategy utility received by player $i$ for each pure strategy $y\in Y$.  The set of weighted potential games is given by
\begin{align*}
\calW := \big\{(u_i)_{i=1}^N \in \R^{NK}:~ \eqref{eq_def_weighted_pot} \mbox{ holds for some } u\in \R^K \mbox{ and some } (w_i)_{i=1}^N,&\\
w_i>0, i=1,~\ldots,N &\big\}.
\end{align*}
the set of exact potential games is given by
\begin{align*}
\calP := \big\{(u_i)_{i=1}^N \in \R^{NK}:~ \eqref{eq_def_weighted_pot} \mbox{ holds for some } u\in \R^K & \mbox{ with } (w_i)_{i=1}^N =\ones \big\},
\end{align*}
where $\ones$ denotes the vector of all ones.

A element $v \in \calP$ or $v\in \calW$ is a vector in $\R^{NK}$ which we decompose as $v=(v_i)_{i=1}^N$
%or $w=(w_i)_{i=1}^N$,
where $v_i\in \R^K$
%and $w_i\in \R^K$
represent the pure-strategy utility of player $i$.
For each $u\in \tilde \calI$, define
\begin{align*}
\calP_u := \big\{v\in \R^{NK}:~ v_i(y_i,y_{-i}) - v_i(y_i',y_{-i}) = u(y_i,y_{-i}) - u(y_i',y_{-i})\\  ~\forall i=1,\ldots,N, ~y_i,y_i'\in Y_i, ~y_{-i}\in Y_{-i} \big\};
\end{align*}
to be the set of all potential games with potential function $u\in \tilde \calI$. Using the definition of an exact potential game (see Section~\ref{sec_pot_games_prelim}), it is straightforward to verify that $\calP_u\cap \calP_{\tilde u} = \emptyset$ for every $u,\tilde u\in \tilde \calI$, $u\not= \tilde u$ and $\bigcup_{u\in \tilde \calI} \calP_u = \calP$. Thus $(\calP_u)_{u\in\tilde \calI}$ partitions $\calP$.
Likewise, define
\begin{align*}
\calW_u := \big\{v\in \R^{NK}:~ v_i(y_i,y_{-i}) - v_i(y_i',y_{-i}) = w_i(u(y_i,y_{-i}) - u(y_i',y_{-i}))\\  ~\forall i=1,\ldots,N, ~y_i,y_i'\in Y_i, ~y_{-i}\in Y_{-i}, \mbox{ and for some } w_i>0 \big\}
\end{align*}
to be set of all weighted potential games with potential function $u\in \tilde \calI$.
As in the case of exact potential games, we see that $(\calW_u)_{u\in \tilde \calI}$ partitions $\calW$.

An equilibrium in a potential game is regular if and only if it is a regular equilibrium in the associated identical payoffs game. Thus, we get the following result.

\begin{lemma} \label{lemma_u_v_equivalence}$~$\\
\noindent (i) For each $u\in\tilde \calI$ and each exact potential game $v\in \calP_u$, $v$ is regular if and only if $u$ is regular.\\
\noindent (ii) For each $u\in\tilde \calI$ and each weighted potential game $w\in \calW_u$, $w$ is regular if and only if $u$ is regular.
\end{lemma}

Recall that $\calP$ and $\calW$ are subspaces of $\R^{NK}$ with dimensions $K_p$ and $K_w$, respectively (see Section~\ref{sec_almost_all}). Since $(\calP_u)_{u\in\tilde \calI}$ partitions $\calP$, every exact potential game $v\in \calP$ may be uniquely represented by a vector $(u,z)$ where $u\in \tilde \calI = \R^{K-1}$ and $z \in \R^{\tilde K_p}$, $\tilde K_p :=K_p-K+1$. Likewise, every weighted potential game $v\in \calW$ may be uniquely represented by a vector $(u,z)$ where $u\in \tilde \calI = \R^{K-1}$ and $z \in \R^{\tilde K_w}$, $\tilde K_w :=K_w-K+1$.

Let $\IR_p\subset \calP$ and $\IR_w\subset \calW$ denote the subsets of the exact and weighted potential games (respectively) that are irregular. Recall that by Remark~\ref{remark_closedness}, $\IR_p$ and $\IR_w$ are closed and hence measurable. Since $\calL^{K-1}(\tilde \calI)= 0$ we see that $\calL^{K_p}(\IR_p) = \calL^{K_p}\left(\tilde \calI \times \R^{K_p-K+1}\right) = \calL^{K-1}\left( \tilde \calI \right) \calL^{K_p-K+1}\left( \R^{K_p-K+1} \right) = 0$.
%Let $h_p,h_w:\tilde \calI\to \R\cup\{\infty\}$ be defined by
%$$h_p(u) := \int \rchi_{\IR_p}(u,z) \,d\calL^{\tilde K_p}(z)$$
%and
%$$h_w(u) := \int \rchi_{\IR_w}(u,z) \,d\calL^{\tilde K_w}(z).$$
%By Lemmas \ref{lemma_reduced_u_regular} and \ref{lemma_u_v_equivalence} we see that $h_p(u) = 0$ and $h_w(u)=0$ for almost every $u\in\tilde \calI$.
%Thus we get
%\begin{align*}
%\calL^{K_p}(\IR_p) & = \int\int \rchi_{IR_p}(u,z) \,d\calL^{\tilde K_p}(z) \,d\calL^{K-1}(u)\\
%& = \int h(u) \,d\calL^{K-1}(u) = 0.
%\end{align*}
By similar reasoning we also see that $\calL^{K_w}(\IR_w) = 0.$
Since $\IR_p$ and $\IR_w$ are closed with respect to $\calP$ and $\calW$ (see Remark~\ref{remark_closedness}), this proves Proposition~\ref{prop_exact_and_weighted}.

\section{Conclusion} \label{sec_conclusion}
Regular NE are isolated, robust, and have a simple analytic structure. Within the subclass of potential games, regular equilibria have a simple and intuitive meaning in terms of the potential function (see Section~\ref{sec_reg_equilib_pot_games}). The paper showed that almost all weighted potential games, almost all exact potential games, and almost all games with identical payoffs are regular.
This result implies that properties holding in regular potential games are inherently robust to payoff perturbations and properties only holding in irregular games are inherently non-robust to payoff perturbations.

Regular potential games can be particularly useful in the study of game-theoretic learning processes. As an example application, in a companion work \cite{swens2017BRdynamics} the authors study BR dynamics in potential games. Regular potential games greatly facilitate the analysis of BR dynamics. Moreover, the results of the present paper allow one to immediately conclude that results obtained for BR dynamics in \cite{swens2017BRdynamics} hold for almost all potential games, thus addressing open questions regarding the convergence rate and limit set composition of BR dynamics \cite{harris1998rate,swens2017BRdynamics,krishna1998convergence,hofbauer1995stability}.
%
%In \cite{swens2017BRdynamics} it was shown that, restricted to the class of regular potential games, BR dynamics are well posed and several important properties can be easily characterized. The present work shows that this characterization of BR dynamics holds in almost all potential games, thus allowing one to address open questions regarding the convergence rate and limit set composition of BR dynamics \cite{harris1998rate,swens2017BRdynamics,krishna1998convergence,hofbauer1995stability}.
As another example, \cite{cohen2017learning} studies a form of no-regret learning in potential games. It is shown that if the potential game is regular, then convergence to NE and convergence rate estimates can be established. The present work allows one to immediately conclude that the results of \cite{cohen2017learning} hold for almost all potential games. An important future research direction is to study the application of regular potential games in facilitating the analysis of additional game-theoretic learning processes.

{\small

\appendix

\section{} \label{Appendix-math-notation}

We use the following standard mathematical notation:
\begin{itemize}
%\item Let $K:= |Y|-N$. (We had it as $K:=|Y|$ before, but I think it will be easier this way since it gives $X\subset \R^K$.
\item $\N:=\{1,2,\ldots\}$.
%\item $\nabla_{x_i}U(x) := (\frac{\partial U}{\partial x_i^k}(x))_{k=1}^{K_i-1}$ gives the gradient of $U$ with respect to the strategy of player $i$ only. $\nabla U(x) := (\frac{\partial U}{\partial x_i^k}(x))_{\substack{i=1,\ldots,N\\ k=1,\ldots,K_i-1}}$ gives the full gradient of $U$.\footnote{NOTE: I think this may be removed!}
\item The mapping $\sgn:\R^{n\times m} \rightarrow \R^{n\times m}$ is given by
$$
\left(\sgn(\vA)\right)_{i,j} :=
\begin{cases}
1 & \mbox{ if } a_{i,j} > 0\\
-1 & \mbox{ if } a_{i,j} < 0\\
0 & \mbox{ if } a_{i,j} = 0.\\
\end{cases}
$$
\item Given two matrices $\vA$ and $\vB$ of the same dimension, $\vA\circ \vB$ denotes the Hadamard product (i.e., the entrywise product) of $\vA$ and $\vB$.
    %That is, if $\vC = \vA\circ \vB$, then  $c_{ij} = a_{ij}b_{ij}$.
\item Given a set of matrices (or scalars) $\{\vA_1,\ldots,\vA_n\}$, possibly of differing dimensions, the notation $\diag(\vA_1,\ldots,\vA_n)$ denotes the associated block diagonal matrix.
\item Suppose $m,n,p\in \N$,  $F_i:\R^m\times \R^n \rightarrow \R$, for $i=1,\ldots,p$. Suppose further that $F: \R^m\times \R^n \rightarrow \R^p$ is given by $F(w,z) = (F_i(w,z))_{i=1,\ldots,p}$. Then the operator $D_w$ gives the Jacobian of $F$ with respect to the components of $w=(w_k)_{k=1,\ldots,m}$; that is
    $$
    D_w F(w,z) =
    \begin{pmatrix}
    \frac{\partial F_1(w,z)}{\partial w_1} \cdots \frac{\partial F_1(w,z)}{\partial w_m}\\
    \vdots \quad \ddots \quad \vdots\\
    \frac{\partial F_p(w,z)}{\partial w_1} \cdots \frac{\partial F_p(w,z)}{\partial w_m}\\
    \end{pmatrix}.
    $$
\item $A^c$ denotes the complement of a set $A$, and $\mathring{A}$ denotes the interior of $A$, and $\cl A$ denotes the closure of $A$.
%\item $A \subset B$ means the set $A$ is contained in the set $B$ (not excluding equality). $A\subsetneq B$ means $A$ is a strict subset of $B$.
%\item The support of a function $f:\Omega\to\R$ is given by $\spt(f):=\{x\in \Omega:f(x)\not= 0\}$.
%\item Given a set $S$ and function $f:S\rightarrow \R$, $\spt(f):=\{x\in S:f(x)\not= 0\}$ denotes the support of $f$.
\item Given a function $f$, $\mathcal{D}(f)$ refers to the domain of $f$ and $\mathcal{R}(f)$ to the range of $f$.
%\item For a set $A\subset \R^n$, the notation $\relint A$ denotes the relative interior of $A$.
\item $\calL^n$, $n\in \{1,2,\ldots\}$ refers to the $n$-dimensional Lebesgue measure.
%\item Given an open set $\Omega \subset \R^n$, and $k\in \N\cup\{\infty\}$, $C_c^k(\Omega)$ denotes the set of $k$-times differentiable functions with compact support in $\Omega$.
\end{itemize}

\section{} \label{Appendix-A-full-rank}
This appendix provides the proof of Proposition~\ref{proposition_A_full_rank}. We will give a short roadmap of the proof before proceeding to the proof itself.
%We will prove the proposition as follows:

\bigskip
\noindent \textbf{Roadmap of Proof of Proposition~\ref{proposition_A_full_rank}.}
\begin{enumerate}
  \item We  begin by reorganizing the columns of $\vA(x)$ and partitioning the matrix as $\vA(x) = [\vA_1(x) ~ \vA_2(x)]$, where $\vA_1(x)$ has dimensions $\gamma\times \tilde K$ with $\tilde K\geq \gamma$ (see \eqref{eq_A_decomposition}).  Since our goal is to prove that $\rank(\vA(x)) = \gamma$ it will be sufficient to prove that $\vA_1(x)$ has full rank.
  \item We will prove Proposition~\ref{proposition_A_full_rank} by considering the sign pattern of $\vA_1(x)$ (i.e, the matrix $\sgn(\vA_1(x))$, where $\sgn(\cdot)$ is defined in\ref{Appendix-math-notation}). To do this, we will leverage known results about sign pattern matrices (i.e., matrices with entries in $\{-1,0,1\}$). We will use the following key notion from the theory of sign pattern matrices (see Definition~\ref{def_L_matrix}): A sign pattern matrix $\vM$ is called an $L$-matrix if every matrix with the same sign pattern as $\vM$ has full rank. Using this notion, proving Proposition~\ref{proposition_A_full_rank} is equivalent to showing that for any $x$ with $\carr(x) \subseteq C$, $\sgn(\vA_1(x))$ is an $L$-matrix (see Lemma~\ref{lemma_A_L_matrix}). Thus, we will focus our efforts on proving Lemma~\ref{lemma_A_L_matrix}.
  \item A convenient characterization of $L$-matrices known from the literature is given in Lemma~\ref{lemma_L_matrix_condition}. We will prove Lemma~\ref{lemma_A_L_matrix} by directly verifying that it satisfies the $L$-matrix condition of Lemma~\ref{lemma_L_matrix_condition}. This verification process will be simplified by breaking $\vA_1(x)$ into component matrices $\vR_1$ and $\vP_1(x)$, where $\vR_1$ is a sign pattern matrix independent of $x$ and $\vP_1(x)$ is a non-negative matrix depending on $x$. The decomposition of $\vA_1(x)$ into $\vR_1$ and $\vP_1(x)$ will be carried out early on in \eqref{def_H_and_P} before proceeding to the rest of the proof.
\end{enumerate}

We now proceed, following the above roadmap.
%Before proceeding to the proof, we will first introduce some notation that will permit us to study the structure of $\vA(x)$.
We begin by assuming a convenient ordering of elements in $Y$.  Let
$$
\tilde{K} := \prod_{i=1}^N \gamma_i.
$$
For $\tau=1,\ldots,\tilde{K}$, let $\alpha_\tau=(\alpha_\tau^1,\ldots,\alpha_\tau^{\tilde{N}})$ be a multi-index associated with the $\tau$-th action tuple in $Y$, meaning that
$$
y^\tau =(y_1^{\alpha_\tau^1},\ldots,y_{\tilde{N}}^{\alpha_\tau^{\tilde{N}}},y_{\tilde{N}+1}^{1},\ldots,y_{N}^{1}),
$$
where $y_i^{1}$, for $i=\tilde{N}+1,\ldots,N$, is, by construction, the pure strategy used by player $i$ at any strategy $x$ with carrier $C$.
%$x^*$ (which has carrier $C$).
Let $Y$ be reordered so that for every $\tau =1,\ldots,\tilde{K}$, $i=1,\ldots,\tilde{N}$ we have $1\leq \alpha_\tau^i \leq \gamma_i$. This ensures that the first $\tilde{K}$ strategies in $Y$ contain all strategy combinations of the elements of $C_1,\ldots,C_N$, and that
for any multi-index $\alpha = (\alpha^1,\ldots,\alpha^{\tilde{N}})$, $1\leq \alpha^i\leq \gamma_i$, there exists a unique $1\leq \tau \leq \tilde{K}$ such that $y^\tau = (y_1^{\alpha^1},\ldots,y_{\tilde{N}}^{\alpha^{\tilde{N}}},y_{\tilde{N}+1}^{1},\ldots,y_{N}^{1}).$

By definition \eqref{def_q_function} and the ordering we assumed for $Y$ we have that
\begin{equation} \label{eq_alpha_q_relationship}
q_i^\tau(y_i^k) =
\begin{cases}
1 & \mbox{ if } k = \alpha_\tau^i\\
0 & \mbox{ otherwise.}
\end{cases}
\end{equation}

For $1\leq s\leq \gamma$ and $1\leq \tau \leq \tilde{K}$, let
\begin{equation} \label{def_H_and_P}
r_{s,\tau} := q^\tau_{i}(\actionikone)- q^\tau_i(\actionione) \quad \mbox{ and } \quad p_{s,\tau}(x) := \prod_{j\not = i} q^\tau_j(x_{j}),
\end{equation}
where $i=i^*(s)$ and $k=k^*(s)$ (see \eqref{def_k_i_star}),
and note that $a_{s,\tau}(x) = r_{s,\tau}p_{s,\tau}(x)$ (see \eqref{def_A}).
We may write $\vA(x) = \vR\circ \vP(x)$, where $\circ$ is the Hadamard product, and $\vR$ and $\vP(x)$ have entries $r_{s,\tau}$ and $p_{s,\tau}(x)$ respectively.

%Using the Hadamard product, $\vA(x)$ may be represented as
%$$\vA(x) = \vR \circ \vP(x).$$
Partition $\vA(x)$, $\vR$ and $\vP(x)$ as
%\footnote{For convenience in notation, we suppress the argument $x$ when writing these partitions.}
$\vA(x) = [\vA_1(x) ~ \vA_2(x)]$, $\vR = [\vR_1 ~ \vR_2]$, and $\vP(x) = [\vP_1(x) ~ \vP_2(x)]$, where $\vA_1(x), \vR_1, \vP_1(x) \in \R^{\gamma \times \tilde{K}}$ and $\vA_2(x), \vR_2, \vP_2(x) \in \R^{\gamma \times (K-\tilde{K})}$, so we may write
\begin{equation} \label{eq_A_decomposition}
\vA(x) = [\vA_1(x) ~ \vA_2(x)], \quad \mbox{ with, } \quad \vA_1(x) = \vR_1\circ \vP_1(x) ~\mbox{ and }~ \vA_2(x) = \vR_2 \circ \vP_2(x).
\end{equation}

In order to show that $\vA(x)$ has full row rank, it is sufficient to prove that $\vA_1(x)$ has full row rank---this is the approach we will take in proving the proposition.

We address this by studying the sign pattern of $\vA_1(x)$. Properties of \emph{sign pattern matrices} (i.e., matrices with entries in $\{-1,0,1\}$) have been well-studied \cite{MR1273974,MR743051}. We recall the following definition from \cite{MR1273974},

\begin{mydef} \label{def_L_matrix}
A sign pattern matrix $\vL\in\R^{m\times n}$ with $n\geq m$ is said to be an $L$-matrix if for every matrix $\vM$ with $\sgn(\vM) = \sgn(\vL)$, the matrix $\vM$ has full row rank.
\end{mydef}

%A characterization of $L$-matrices is given in the following Lemma \cite{MR1273974,MR743051}.

The following lemma characterizes $L$-matrices \cite{MR1273974,MR743051}.
\begin{lemma}\label{lemma_L_matrix_condition}
Let $\vL\in\R^{m\times n}$ be a sign pattern matrix with $n\geq m$.
Then $\vL$ is an $L$-matrix if and only if for every diagonal sign pattern matrix $\vD\in\R^{m\times m}$, $\vD\not=0$ there is a nonzero column of $\vD\vL$ in which each nonzero entry has the same sign.
\end{lemma}

In light of Definition~\ref{def_L_matrix}, Proposition~\ref{proposition_A_full_rank} is equivalent to the following lemma.

\begin{lemma} \label{lemma_A_L_matrix}
For any $x$ such that $\carr(x) \subseteq C$, the matrix $\vA_1(x) = $ \newline $ \left(\vR_1 \circ \sgn(\vP_1(x))\right)$ is an $L$-matrix.
\end{lemma}
The proof of this lemma relies on showing that $\left(\vR_1 \circ \sgn(\vP_1(x))\right)$ satisfies the $L$-matrix characterization given in Lemma~\ref{lemma_L_matrix_condition}.
%While this lemma is critical to the proof of Proposition \ref{prop_identical_payoff}, the proof of the lemma requires one to check several special cases and can be somewhat tedious. Readers may which to skip the proof of this lemma on a first read through.
%While this lemma is critical to the proof of Proposition \ref{prop_identical_payoff}, the proof of the lemma is somewhat tedious and readers may wish to skip the proof on first read through. The complete proof of Lemma \ref{lemma_A_L_matrix} can be found in \ref{apdx_L_lemma}.

Before proving Lemma~\ref{lemma_A_L_matrix} we introduce some definitions that will be useful in the proof.

%Suppose $\vD\in\R^{\gamma\times\gamma}$ is a diagonal sign pattern matrix.
Given a diagonal matrix $\vD\in\R^{\ell\times\ell}$, $\ell\in\N$,
let $\diag(\vD)$ be the vector in $\R^\ell$ containing the diagonal elements of $\vD$.

Given a diagonal sign pattern matrix $\vD\in\R^{\ell\times\ell}$, $\ell\in \N$, define $\idx(\vD)$ as follows. If $\diag(\vD)$ does not contain any ones, then let $\idx(\vD) = 1$. Otherwise, let $\idx(\vD)$ be one more than the first index in $\diag(\vD)$ containing a 1.\footnote{Assume indexing starts with one, not zero. For example, if the first time a 1 appears in $\diag(\vD)$ is at index $2$, then $\idx(\vD) = 3$.}\footnote{The awkward offset in this definition is needed in order to handle the indexing offset inherent in the definition of $X_i$, as discussed in Section~\ref{sec:prelim-1}.}

Given a diagonal matrix $\vD\in\R^{\gamma\times\gamma}$,
let $\vD_i\in\R^{(\gamma_i-1)\times(\gamma_i-1)}$, $i=1,\ldots,\tilde{N}$ be the (unique) diagonal matrices satisfying $$\diag(\vD) = (\diag(\vD_1),\ldots,\diag(\vD_N)).$$

We now prove Lemma~\ref{lemma_A_L_matrix}.
\begin{proof}
Let $x$ be a strategy satisfying $\carr(x)\subseteq C$. In order to show that
\newline $\left(\vR_1 \circ \sgn(\vP_1(x))\right)$ is an $L$-matrix, it is sufficient (by Lemma~\ref{lemma_L_matrix_condition}) to show that and for any diagonal sign pattern matrix $\vD\not= 0$, there exists a column of $\vD (\vR_1 \circ \sgn(\vP_1(x)))$ which is nonzero and in which every nonzero entry has the same sign.
With this in mind, we begin by giving a characterization of the columns of $\vR_1$.

Suppose that $i=1,\ldots,\tilde{N}$ and $k=1,\ldots,\gamma_i-1$ are fixed. Note the following:
\begin{enumerate}[label=(\roman*)]
\item  Suppose $\tau\in \{1,\ldots,\tilde{K}\}$ is such that $\alpha_\tau^i = k+1$, where $\alpha_\tau^i$ is the $i$-th index of the multi-index $\alpha_\tau$. Since $\alpha_\tau$ is used to define the ordering of actions in $Y$, we have $q_i^\tau(\actionikone) = 1$ and $q_i^\tau(\actionione) = 0$ (see \eqref{eq_alpha_q_relationship} and preceding discussion). Hence, $q^\tau_i(\actionikone)- q^\tau_i(\actionione) = 1$.
\item  Suppose $\tau\in \{1,\ldots,\tilde{K}\}$ is such that $\alpha_\tau^i = 1$. Then $q_i^\tau(\actionikone) = 0$, and $q_i^\tau(\actionione) = 1$. Hence, $q^\tau_i(\actionikone)- q^\tau_i(\actionione) = -1$.
\item  For all other $\tau\in \{1,\ldots,\tilde{K}\}$ we have $q_i^\tau(\actionikone) = 0$, and $q_i^\tau(\actionione) = 0$, and hence $q^\tau_i(\actionikone)- q^\tau_i(\actionione) = 0$.\\
\end{enumerate}

For $1\leq \tau \leq \tilde{K}$, let $\vr_\tau\in\R^\gamma$ be the $\tau$-th column of $\vR_1$. Partition this column as
$$
\vr_\tau =
\begin{pmatrix}
\vr_\tau^1\\
\vdots\\
\vr_\tau^{\tilde{N}}
\end{pmatrix}.
%((r_\tau^1)^T ~\cdots~ (r_\tau^{\tilde{N}})^T)^T,
$$
where $\vr_\tau^i \in \R^{\gamma_i-1}$. From \eqref{def_H_and_P} we see that
$$
\vr_\tau^i =
\begin{pmatrix}
q_i^\tau(\actionitwo) - q_i^\tau(\actionione)\\
\vdots\\
q_i^\tau(\actionigammai) - q_i^\tau(\actionione)
\end{pmatrix}.
$$
Given the observations (i)--(iii) above we see that for each $i$ we have
%Recall that $r_\tau$ may be decomposed as $r_\tau = ((r_\tau^1)^T,\ldots,(r_\tau^{\tilde{N}})^T)^T$, where $r_\tau^i \in \R^{\gamma_i-1}$ is given by
\begin{equation} \label{def_r_tau}
\vr_\tau^i =
\begin{cases}
-\ones & \mbox{ if } \alpha_\tau^i = 1\\
e_{\alpha_\tau^i-1} & \mbox{ if } 2 \leq \alpha_\tau^i \leq \gamma_i,
\end{cases}
\end{equation}
where the symbol $e_{\alpha_\tau^i-1}$ refers to the $(\alpha_\tau^i-1)$-th canonical vector in $\R^{\gamma_i-1}$ and $\ones \in \R^{\gamma_i-1}$ is the vector of all ones.

We now characterize the columns of $\sgn(\vP_1(x))$.
%Given an $x=(x_1,\ldots,x_{N})$ with $\carr(x) \subseteq C$,
For $i=1,\ldots,\tilde{N}$ we define
$$
\idx_i(x) := \big\{k\in\{1,\ldots,\gamma_i\}:~ T_i^k(x_i)>0\big\}.
$$
The function $\idx_i(x)$ simply gives the indices of the pure strategies of player $i$ that receive positive weight under $x$.
Since $\carr(x) = C$, the ordering we assumed for $Y_i$ implies that $T_i^k(x) = 0$ for $k\geq \gamma_i+1$.
By the definition of $T_i^k$, it is not possible to have $T_i^k(x_i)=0$ for all $k=1,\ldots,K_i$ and hence $\idx_i(x) \not=\emptyset$.
%$\subsetneq \{1,\ldots,\gamma_i\}$.

Let $\vp_\tau$ be the $\tau$-th column of $\vP_1(x)$ and let
$$
\tilde \vp_\tau := \sgn(\vp_\tau)
$$
be the $\tau$-th column of $\sgn(\vP_1(x))$.\footnote{For ease of notation we suppress the argument $x$ when writing the columns of these matrices.}
Suppose that $\tau\in\{1,\ldots,\tilde{K}\}$ is such that for the multi-index $\alpha_\tau$ we have $\alpha_\tau^i \in \idx_i(x)$
%changed
for all $i=1,\ldots,\tilde{N}$. Then for each $s=1,\ldots,\gamma$ the $(s,\tau)$-th entry of $\vP_1(x)$ is strictly positive (see \eqref{def_q_function} and \eqref{def_H_and_P}), and hence $\vp_\tau$ is positive and $\tilde \vp_\tau = \ones$.

Partition the columns $\vp_\tau$ and $\tilde \vp_\tau$ as
\begin{equation} \label{def_p_partition}
\vp_\tau =
\begin{pmatrix}
\vp_\tau^1\\
\vdots\\
\vp_\tau^{\tilde{N}}
\end{pmatrix}
\quad \mbox{ and } \quad
\tilde \vp_\tau =
\begin{pmatrix}
\tilde \vp_\tau^1\\
\vdots\\
\tilde \vp_\tau^{\tilde{N}}
\end{pmatrix},
\end{equation}
where $\vp_\tau^i,\tilde \vp_\tau^i \in \R^{\gamma_i-1}$.
Suppose that $\tau$ is such that for the multi-index $\alpha_\tau$ we have $\alpha_\tau^i \notin \idx_i(x)$
%changed
for exactly one subindex $i\in\{1,\ldots,\tilde{N}\}$.
Then $\vp_\tau^i$ is positive (see \eqref{def_H_and_P}) and $\vp_\tau^j$ is zero for any $j\not=i$. Hence, $\tilde \vp_\tau^i = \ones$ and $\tilde \vp_\tau^j=0$ for any $j\not=i$.

Now, let $\vD \in \R^{\gamma\times \gamma}$ be an non-zero diagonal sign pattern matrix. We will show that there is a nonzero column of $\vD (\vR_1\circ \sgn(\vP_1(x)))$ in which each nonzero entry is a 1.

We now consider two possible cases for the structure of $\vD$ and show that in each case there is a nonzero column of $\vD(\vR_1\circ\sgn(\vP_1(x)))$ in which every nonzero entry is 1.

\textbf{Case 1}: Suppose that for all $i\in\{1,\ldots,\tilde{N}\}$ such that $\idx(\vD_i) \notin \idx_i(x)$
%changed
we have $\diag(\vD_i) = 0$.
Choose $\tau$ such that
$$
\begin{cases}
\alpha_\tau^i = \idx(\vD_i) & \mbox{ if }~ \idx(\vD_i) \in \idx_i(x),\\
\alpha_\tau^i\in \idx_i(x) & \mbox{ if } ~ \idx(\vD_i) \notin \idx_i(x).
\end{cases}
$$
%changed
Note that $\alpha_\tau^i \in \idx_i(x)$
%changed
for all $i=1,\ldots,\tilde{N}$, and hence $\tilde \vp_\tau = \ones$ (see discussion preceding \eqref{def_p_partition}).
The $\tau$-th column of $\vD(\vR_1\circ \sgn(\vP_1(x)))$ is given by
\begin{equation} \label{eq_tau-th_column}
\vD (\vr_\tau \circ \tilde \vp_\tau) = \diag(\vD) \circ \vr_\tau  = \left(\diag(\vD_i) \circ \vr_\tau^i\right)_{i=1}^{\tilde{N}}.
\end{equation}
%So to show that the $\tau$-th column of  $\vD (\vR \circ \tilde \vP)$ is nonzero and each nonzero entry has the same sign
%it is sufficient to show that $(\diag(\vD_i) \circ \vr_\tau^i)_{i=1}^{\tilde{N}}$ is nonzero and each nonzero entry is 1.

For any $i$ such that $\idx(\vD_i) \notin \idx_i(x)$
%changed
we have, by assumption, $\diag(\vD_i)= 0$ and hence $\diag(\vD_i) \circ \vr_\tau^i = 0$. Moreover, note that in this case we have $\idx(\vD_i) = 1$ since, by the definition of $\idx(\cdot)$, $\diag(\vD_i)=0$ implies $\idx(\vD_i) = 1$.

Suppose now that $i$ is such that $\idx(\vD_i)\in \idx_i(x)$.
%changed
For $i=1,\ldots,\tilde{N}$, if $\alpha_\tau^i = 1$ then $\vr_\tau^i = -\ones$ (by \eqref{def_r_tau}) and $\diag(\vD_i)$ contains no ones (this is the definition of $\idx(\vD_i) = 1$). In fact, $\diag(\vD_i)$ contains only entries with value of 0 or $-1$. Hence, $\diag(\vD_i) \circ \vr_\tau^i = -\diag(\vD_i)$, which is a nonnegative vector.

If $2 \leq \alpha_\tau^i \leq \gamma_i$ then $\vr_\tau^i = e_{\alpha_\tau^i-1}$ (by \eqref{def_r_tau}). Recalling the definition of $\idx(\cdot)$, by our choice of $\alpha_\tau^i = \idx(\vD_i)$, the $(\alpha_\tau^i-1)$-th entry of $\diag(\vD_i)$ is 1. Hence, $\diag(\vD_i) \circ \vr_\tau^i  = e_{\alpha_\tau^i-1}$.
In particular, this implies that if $\alpha_\tau^i\not=1$ then $\diag(\vD_i) \circ r_\tau^i$ is not identically zero and every nonzero entry of $\diag(\vD_i) \circ r_\tau^i$ is 1.

%The $\tau$-th column of $\vD\vR$ is given by $\vD\vr_\tau = \diag(\vD) \circ \vr_\tau = (\diag(\vD_i) \circ \vr_\tau^i)_{i=1}^{\tilde{N}} \circ \tilde \tilde \vp_\tau$.
%We have shown that for each $i$, every nonzero entry of $\diag(\vD_i) \circ \vr_\tau^i$ has the value 1.

In summary, for $i=1,\ldots,\tilde{N}$, we have $\diag(\vD_i)\circ \vr_\tau^i \geq 0$, with equality only when $\idx(\vD_i) = 1$ and $\vD_i = 0$. Hence, by \eqref{eq_tau-th_column}, the $\tau$-th column of $\vD(\vR_1\circ \sgn(\vP_1(x)))$ satisfies $\vD(\vr_\tau \circ \tilde \vp_\tau) \geq 0$, with equality only when $\idx(\vD_i) = 1$ and $\vD_i = 0$ for all $i$. But, by assumption $\vD \not= 0$, so  $\vD(\vr_\tau \circ \tilde \vp_\tau)\not=0$.

\textbf{Case 2}: Suppose that for some $i\in\{1,\ldots,\tilde{N}\}$ we have
$\idx(\vD_i) \not\in \idx_i(x)$
%changed
and $\diag(\vD_i)\not= 0$.
Let $\tau\in\{1,\ldots,\tilde{K}\}$ be chosen such that $\alpha_\tau^i = \idx(\vD_i)$ for exactly one such $i\in\{1,\ldots,\tilde{N}\}$ and for all other $j\not= i$ we have $\alpha_\tau^i\in \idx_j$
%changed
(this is always possible since $\idx_j \not=\emptyset$).
%changed
Then we have
$$
\tilde \vp_\tau^i = \ones, \quad\quad \mbox{ and } \quad\quad \tilde \vp_\tau^j = 0,~ \mbox{ for all } j\not=i
$$
(see discussion following \eqref{def_p_partition}).

As shown in Case 1, if $\alpha_\tau^i = 1$, then $\vD_i \leq 0$ and $r_\tau^i = -1$ which implies that $\diag(\vD_i) \circ \vr_\tau^i \geq 0$.
%for all $i=1,\ldots,\tilde{N}$.
Moreover, since $\tilde \vp_\tau^i = \ones$ and since, by assumption $\diag(\vD_i)\not = 0$ we have $\diag(\vD_i) \circ \vr_\tau^i \circ \tilde \vp_\tau^i \not= 0$ and every nonzero entry is 1.

If $2\leq \alpha_\tau^i \leq \gamma_i$, then, again using the same reasoning as in Case 1, we see that $\diag(\vD_i) \circ \vr_\tau^i  = e_{\alpha_\tau^i-1}$. Since $\tilde \vp_\tau^i = \ones$ we get that $\diag(\vD_i) \circ \vr_\tau^i \circ \vp_\tau^i = e_{\alpha_\tau^i-1}$.

For $j\not= i$ we have $\tilde \vp_\tau^j = 0$, %so regardless of the particular choice of $\alpha_\tau^i$
which implies $\diag(\vD_j) \circ \vr_\tau^j \circ \tilde \vp_\tau^j = 0$.

All together, this implies that the $\tau$-th column of $\vD(\vR_1 \circ \sgn(\vP_1(x)))$, given by $(\diag(\vD_j)\circ \vr_\tau^j \circ \tilde \vp_\tau^j)_{j=1}^{\tilde{N}}$, is nonzero and every nonzero entry is equal to 1.

Since this holds for arbitrary diagonal sign matrix $\vD\not=0$, Lemma~\ref{lemma_L_matrix_condition} implies that $(\vR_1 \circ \sgn(\vP_1(x)))$ is an $L$-matrix. Since this holds for any $x$ satisfying $\carr(x) \subseteq C$, we see that the desired result holds.
%Since $\vA_1$ has the same sign pattern as $(\vR_1 \circ \sgn(\vP_1(x)))$, it has full rank.
\end{proof}

\section{} \label{Appendix-general-extras}

\begin{lemma}
Let $\vB(x)$ be as defined in Section~\ref{sec:gen-games-proof-strategy}. There are $\gamma$ columns of $\vB(x)$ such that these columns form an invertible diagonal matrix for any $x$ with $\carr(x) = C$.
\end{lemma}
\begin{proof}
Without loss of generality, suppose $C=C_1\cup\cdots\cup C_N$ and assume that each player's pure strategy set $Y_i$ is ordered so that $y_i^1\in C_i$, $i=1,\ldots,\tilde N$.
Now, let $i\in\{1,\ldots,\tilde N\}$ be arbitrary. Suppose that the set of pure strategies $Y$ is ordered and that the $\tau$'th pure strategy, $\tau\in\{1,\ldots,K\}$, takes the form $(y_1^1,\ldots,y_{i-1}^1,y_i^{k+1},y_{i+1}^1,\ldots,y_N^1)$, where $i\in \{1,\ldots,\tilde N\}$ and $k\in \{1,\ldots,\gamma_i-1\}$. For a given $i$, there are precisely $\gamma_i-1$ such strategy profiles.
Fixing $\tau$ to this value (thus looking at a single column of $\vB(x)$) we see from the definition of $q_i^\tau$ that
$$
q_i^\tau(y_i^\ell) =
\begin{cases}
1 & \text{if } \ell = k+1\\
0 & \text{else}.
\end{cases}
$$
Given our choice of $\tau$ (and ordering of $Y_i$) we also have $\prod_{j\not= i} q_j^\tau(x_j)>0$ for all $x$ with carrier $C$. Given the definition of $\vB(x)$ in \eqref{def_B}, this implies that the $\tau$-th column of $\vB_i(x)$ is composed of zeros except for the $k$-th entry which takes the positive value $\prod_{j\not= i} q_j^\tau(x_j)>0$.

Without loss of generality, we may reorder the set of pure strategies so that the $\gamma$ combinations of pure strategies taking the aforementioned form constitute the first $\gamma$ strategies in the set. Under this ordering, we have shown that $\vB(x)$ takes the form $\vB(x) = [\vC(x)~ \vD(x)]$, where $\vC(x)$ is a $\gamma\times \gamma$ diagonal matrix (up to a permutation of the columns) with diagonal entries $\prod_{j\not= i} q_j^\tau(x_j)>0$.
\end{proof}

%\bigskip
%\bigskip
%\noindent\textbf{Some Auxiliary Lemmas.}

\begin{lemma} \label{lemma_apx_BR_vs_gradient}
Let $x\in X$ and $i=1,\ldots,N$. Assume $Y_i$ is ordered so that $\actionione\in BR_i(x_{-i})$. Then: (i) For $k=1,\ldots,K_i-1$ we have $\frac{\partial U(x)}{\partial x_i^k}\leq 0$, and (ii) For $k=1,\ldots,K_i-1$, we have $\actionikone \in BR_i(x_{-i})$ if and only if $\frac{\partial U(x)}{\partial x_i^k}=0$. In particular, combined with (i) this implies that $\actionikone \not\in BR_i(x_{-i}) \iff \frac{\partial U(x)}{\partial x_i^k}<0$.
\end{lemma}
\begin{proof}
(i) Differentiating \eqref{eq_potential_expanded_form2} we find that
$\frac{\partial U(x)}{\partial x_i^{k}} = U(\actionikone,x_{-i}) - U(\actionione,x_{-i})$
%\begin{align} \label{eq_potential_derivative}
%\frac{\partial U(x)}{\partial x_i^{k}} = U(\actionikone,x_{-i}) - U(\actionione,x_{-i}).
%\end{align}
(i)  Since $\actionione$ is a best response, we must have $U(\actionione,x_{-i}) \geq U(\actionikone,x_{-i})$ for any $k=1,\ldots,K_i-1$. Hence $\frac{\partial U(x)}{\partial x_i^{k}}\leq 0$.\\
(ii) Follows readily from \eqref{eq_potential_expanded_form2}.
\end{proof}

%\begin{lemma} \label{lemma_eq_BR_positive_gradient}
%Let $x \in X$. If $\actionik \in BR_i(x_{-i})$ then $\frac{\partial U(x)}{\partial x_i^{k}} \geq 0$.
%\end{lemma}
%\begin{proof}
%The result follows readily from \eqref{eq_potential_derivative}.
%\end{proof}
%
\begin{lemma} \label{lemma_carrier_vs_gradient}
Suppose $x^*$ is an equilibrium and $y_i^k \in \carr(x^*)$, $k\geq 2$. Then $\frac{\partial U(x^*)}{\partial x_i^k} = 0$.
\end{lemma}
\begin{proof}
Since $U$ is multilinear, $y_i^k$ must be a pure-strategy best response to $x_{-i}^*$. The result then follows from Lemma~\ref{lemma_apx_BR_vs_gradient}.
\end{proof}

}

%\section{} \label{Appendix-general-extras}


%\section*{References}
\bibliographystyle{elsarticle-num}
\bibliography{myRefs}




%\bibitem{}
%
%\end{thebibliography}
\end{document}